\documentclass[11pt]{article}
\usepackage{tikz}
\usetikzlibrary{calc,positioning,shapes.geometric}
\usepackage{pgfplots}
\pgfplotsset{compat=1.18}
\usepackage[margin=1.1in]{geometry}
\usepackage{amsmath, amssymb, amsthm}
\usepackage{url}
\usepackage[numbers,sort&compress]{natbib}
\usepackage[hidelinks]{hyperref}
\newcommand{\doi}[1]{\href{https://doi.org/#1}{\nolinkurl{https://doi.org/#1}}}
\DeclareRobustCommand{\BallIP}{\textnormal{\textsc{BallIP}}}
\DeclareRobustCommand{\BallIPp}[1]{\ensuremath{\text{\BallIP}_{#1}}}
\usepackage{enumitem}
\usepackage{booktabs,tabularx,array}
\usepackage{float}
\usepackage{needspace}

\theoremstyle{plain}
\newtheorem{theorem}{Theorem}[section]
\newtheorem{proposition}[theorem]{Proposition}
\newtheorem{lemma}[theorem]{Lemma}
\newtheorem{corollary}[theorem]{Corollary}

\theoremstyle{definition}
\newtheorem{problem}[theorem]{Problem}

\theoremstyle{remark}
\newtheorem{remark}[theorem]{Remark}

\DeclareMathOperator{\OPT}{OPT}
\newcommand{\Z}{\mathbb{Z}}
\newcommand{\R}{\mathbb{R}}
\newcommand{\Q}{\mathbb{Q}}
\newcommand{\cmax}{c_{\max}}

\newcommand{\xc}{\vec{x}^{c}}
\renewcommand{\vec}[1]{\mathbf{#1}}
\newcommand{\norm}[1]{\lVert #1 \rVert}

\title{Integer Maximization over $\ell_p$ Balls:\\
  Hardness and Exact Algorithms}
\author{Cinar Ari \qquad Robert Hildebrand\\[2pt]
  \small Grado Department of Industrial and Systems Engineering, Virginia Tech}
\date{September 5, 2026}

\begin{document}
\maketitle

\begin{abstract}
We study the problem of maximizing a linear function over the integer points
of an origin-centered $\ell_p$ ball, which we call \BallIPp{p}. For every fixed integer $p\ge2$, we prove that the decision
problem over an $\ell_p$-ball is NP-complete.

We then focus on the Euclidean case and study how the difficulty of the problem depends on the numerical parameters of the instance. We give
two complementary pseudo-polynomial exact algorithms. The first specializes
a known radius-budget dynamic program; the same nonlinear-knapsack framework
also applies to every fixed finite integer $p$. We then develop a complementary
dynamic program over candidate objective values for the Euclidean case. The latter is polynomial in
the encoding size of the radius and pseudo-polynomial in the magnitude of the
cost coefficients.

For a fixed
objective value, feasibility can be formulated as a closest vector problem (CVP) instance. This connection
gives an exact algorithm whose running time depends on the covering radius of
that lattice. Conversely, we
show that a rank-$k$ Euclidean closest-vector instance in ambient dimension $d$ reduces
to the decision version of a Euclidean \BallIP{} instance in dimension $d+k\le2d$,
transferring known bounds under the Exponential Time Hypothesis (ETH). Finally, we study
dimension-dependent approaches based on proximity to the continuous optimizer
and describe the fixed-level
constructions that extend to ellipsoids.
\end{abstract}

\section{Introduction}\label{sec:intro}

\subsection{The problem}\label{sec:intro-problem}

We study the family of problems obtained by maximizing a linear objective over
the integer points of an origin-centered $\ell_p$ ball. For an integer
$p\ge1$, or for $p=\infty$, a cost vector $\vec c\in\Z^n$, and a radius
$r\ge0$, define
\begin{equation}\label{eq:intro-ballip-signed}
    \BallIPp{p}:\qquad
    \max\left\{
        \vec c^\top\vec x:
        \vec x\in\Z^n,\ \norm{\vec x}_p\le r
    \right\}.
\end{equation}
Its decision version also takes a threshold $\gamma\in\Z$ and asks whether a
feasible integer point has objective value at least $\gamma$. The endpoint
cases $p=1$ and $p=\infty$ are trivially solvable in polynomial time, so,
apart from recording them for completeness, we henceforth restrict attention
to finite $p\ge2$. Due to encoding issues arising from irrational numbers
for nonintegral exponents, we consider only integer $p$. For a fixed integer
exponent $p\ge2$, our computational formulation encodes the radius by the
integer $p$-th-power budget $\rho=r^p\in\Z_{\ge0}$, so the constraint is
$\sum_i |x_i|^p\le\rho$. We reserve the unsubscripted notation \BallIP{} for
the Euclidean case $p=2$, where the input radius is encoded by
$r^2\in\Z_{\ge0}$.

The problems $\BallIPp{p}$ are among the simplest convex integer optimization
problems: the objective is linear, the feasible region is a symmetric norm
ball, the underlying lattice is $\Z^n$, and the continuous optimum is explicit. The norm
constraint is separable and resembles a nonlinear knapsack budget. These observations leave open, a priori, whether this restricted
problem admits a polynomial-time algorithm, or whether it is NP-hard.

\subsection{Main results}\label{sec:intro-results}

\paragraph{Hardness.}
For every fixed integer $p\ge2$, the decision version of $\BallIPp{p}$ is
NP-complete. This settles the norm-ball optimization question arising in
Ley and Merkert~\cite{leymerkert} for these exponents.
Having established this norm-wide hardness result, we focus primarily on the
Euclidean problem \BallIP{}.

\paragraph{Pseudo-polynomial algorithms.}
For every fixed finite integer $p\ge2$, $\BallIPp{p}$ can be viewed as a
highly structured separable nonlinear-knapsack problem after the sign
normalization. Choosing $x_i=m\ge0$ earns $c_i m$ units of objective value and
uses $m^p$ units of the $p$-th-power radius budget. The marginal budget costs
$(m+1)^p-m^p$ are the same for every coordinate, so the heterogeneous item
weights present in ordinary Knapsack are absent here.

The exact dynamic program of Hochbaum~\cite{hochbaum1995} for integer
nonlinear knapsack therefore applies to every such fixed $p$. In the Euclidean
case it yields the $O(nr^3)$ radius-budget dynamic program that we state below.
We also develop a second pseudo-polynomial dynamic program (polynomial in the
numerical values of the input rather than in their bit lengths), specific to the
Euclidean case, that works over candidate objective values.
It indexes states by objective value rather than by radius budget and solves
\BallIP{} in $O(n^4\cmax^4)\operatorname{poly}(L)$ bit time, where $\cmax=\norm{\vec c}_\infty$ and
$L=\langle\vec c,r^2\rangle$ is the total binary encoding length of the input. In particular, its dependence on the squared
radius is polynomial in the binary encoding length, while its dependence on
the coefficients of $\vec c$ is pseudo-polynomial. Thus \BallIP{} is
polynomial-time solvable whenever either the squared-radius budget or the cost
coefficients are numerically small.

\paragraph{Closest vector problem and \BallIP.}
The objective-value dynamic program suggests a different way to view
\BallIP. For this discussion, we use the preprocessing from Section~\ref{sec:notation}, so $n\ge2$, $r>0$,
and the costs are positive and primitive, meaning $\gcd(c_1,\ldots,c_n)=1$. Rather than looking for a feasible integer point in the ball that
maximizes the objective, we may ask, for a fixed value $m$, whether there
exists an integer point satisfying
\[
    \vec c^\top\vec x=m
    \qquad\text{and}\qquad
    \norm{\vec x}\le r.
\]
The integer points on the hyperplane $\vec c^\top\vec x=m$ form a coset of
\[
    \Lambda_0
    =
    \{\vec x\in\Z^n:\vec c^\top\vec x=0\}.
\]
Testing whether a fixed objective level contains an integer point in the ball
can therefore be formulated as an instance of the closest vector problem
(CVP), which asks for the lattice point closest to a given target. This fixed-level formulation does not by itself give a polynomial-time
algorithm for \BallIP{} by scanning all objective levels, since their number
can be exponential in the binary input length.

The geometry of these levels nevertheless gives an exact algorithm whose
running time can be controlled by the covering radius. Let $\mu$ denote the
covering radius of $\Lambda_0$, the largest distance from a point of its span
to the lattice (Section~\ref{sec:notation}). If $r\ge\mu$, we obtain an exact algorithm with
running time
\[
    2^{O(n)}
    \left(1+\frac{\norm{\vec c}\mu^2}{r}\right)
    \operatorname{poly}(L).
\]
In particular, this is $2^{O(n)}\operatorname{poly}(L)$ whenever
$\norm{\vec c}\mu^2/r\le2^{O(n)}$ with a uniform constant in the exponent.

We also prove a reduction in the reverse direction. An arbitrary Euclidean
CVP instance in ambient dimension $d$ and lattice rank $k$ (the number of
basis vectors) reduces to a
\BallIP{} instance in dimension $ d+k\le2d.$ Combining this reduction with known lower bounds for Euclidean CVP based on
the Exponential Time Hypothesis (ETH), which asserts that 3-SAT cannot be
solved in $2^{o(n)}$ time, implies that \BallIP{} does not admit a$
    2^{o(n)}
    \operatorname{poly}(\langle\vec c,r^2\rangle)$
exact algorithm unless ETH fails~\cite[Theorem~1.1]{AGM25}. Thus
dimension-controlled improvements for \BallIP{} would also transfer to
Euclidean CVP.

\paragraph{Other results.}
Since the continuous optimizer is explicit, one natural approach is to
enumerate integer points near it. We derive coordinate bounds for such a
search and show that, although they improve on naive enumeration, explicit
enumeration of the resulting boundary cap, the part of the ball above the
objective threshold within which every optimal point lies, can still have radius-dependent
size. We also identify which fixed-objective-level constructions extend from
the Euclidean ball to integer optimization over a general ellipsoid.

Table~\ref{tab:results} summarizes the main complexity and algorithmic results.

\begin{table}[H]
\centering
\footnotesize
\renewcommand{\arraystretch}{1.06}
\begin{tabularx}{\textwidth}{@{}>{\raggedright\arraybackslash}p{0.31\textwidth}>{\raggedright\arraybackslash}X@{}}
\toprule
\textbf{Result} & \textbf{Statement} \\
\midrule
\multicolumn{2}{@{}l}{\textbf{Hardness}}\\[-1pt]
\addlinespace[1pt]
Fixed integer $p\ge2$
& The decision version of $\BallIPp{p}$ is NP-complete (Theorem~\ref{thm:lp}). \\
\midrule
\multicolumn{2}{@{}l}{\textbf{Exact algorithms for \BallIP}}\\[-1pt]
\addlinespace[1pt]
Knapsack DP
& Exact in $O(nr^3)$ arithmetic operations on polynomial-bit integers, hence $O(nr^3)\operatorname{poly}(L)$ bit time (Proposition~\ref{prop:dp}). \\
\addlinespace[1pt]
Candidate objective values
& Exact in $O(n^4\cmax^4)\operatorname{poly}(L)$ bit time; polynomial in the encoding length of $r^2$ and pseudo-polynomial in the objective coefficients (Theorem~\ref{thm:window}). \\
\midrule
\multicolumn{2}{@{}l}{\textbf{CVP and other results}}\\[-1pt]
\addlinespace[1pt]
Covering-radius algorithm
& If $r\ge\mu$, exact in $2^{O(n)}(1+\norm{\vec c}\mu^2/r)\operatorname{poly}(L)$ time; in particular single-exponential when $\norm{\vec c}\mu^2/r\le2^{O(n)}$ uniformly (Theorem~\ref{thm:covscan}). \\
\addlinespace[1pt]
CVP reduction
& CVP in dimension $d$ and rank $k$ reduces to \BallIP{} in dimension $d+k\le2d$; this yields the ETH lower bound (Theorem~\ref{thm:main}, Remark~\ref{rem:cvp-ballip-hardness}). \\
\addlinespace[1pt]
Boundary enumeration
& Gives an exact radius-dependent enumeration algorithm (Theorem~\ref{thm:enumeration}). \\
\bottomrule
\end{tabularx}
\caption{Summary of the main results.}
\label{tab:results}
\end{table}

\subsection{Motivation and related work}\label{sec:intro-related}

\paragraph{Integral inverse optimization.}
Our main motivation for \BallIP{} and its $\ell_p$ generalization comes from
integral inverse optimization. Inverse optimization modifies objective
coefficients so that a given feasible solution becomes optimal; see Ahuja and
Orlin~\cite{ahujaorlin} and Chan, Mahmood, and Zhu~\cite{chan}.
Ley and Merkert~\cite[Section~3]{leymerkert} require these modifications to be
integral. Their analysis leads to optimization over integer points of
origin-centered norm balls, which we study for fixed integer $p\ge2$.
Their coNP-membership result for the inverse problem under $\ell_1$ and
$\ell_\infty$ uses polynomial-time optimization over the corresponding
balls~\cite[Theorem~3.6]{leymerkert}. Our hardness result rules out the same
optimization approach for fixed integer $p\ge2$, unless $\mathrm{P}=\mathrm{NP}$;
it does not resolve coNP membership for the inverse problem itself.

\paragraph{Allocation problems and nonlinear knapsack.}
\BallIP-type constraints also arise in structured allocation models.
Round-lot mean--variance portfolio selection has integer holdings and a
quadratic risk constraint~\cite{lisunwang,corazzafavaretto,bonamilejeune};
equal diagonal covariance coefficients give a Euclidean ball. Likewise,
integer speed levels in an idealized homogeneous power-allocation model with
power proportional to $s^\alpha$ give
$\sum_i x_i^\alpha\le B$~\cite{isci2006,yds95,bkp07}.
These are symmetric special cases, not the full applied formulations.

More generally, \BallIP{} is a separable nonlinear-knapsack problem within the
resource-allocation literature~\cite{fox1966,hochbaum1990,bretthauer2002}.
Hochbaum's exact dynamic program and fully polynomial-time approximation scheme
(FPTAS) apply here~\cite[Sections~1 and~3]{hochbaum1995}; we specialize the
dynamic program in Section~\ref{sec:exact}. Ordinary Knapsack hardness does
not settle the complexity of this restriction: the marginal radius costs
$1,3,5,\ldots$ are identical across coordinates.

\paragraph{Convex integer programming and trust-region problems.}
\BallIP{} is a special case of convex integer optimization.  The classical fixed-dimensional theory begins with Lenstra~\cite{lenstra}, who showed that integer linear programming is polynomial-time solvable in fixed dimension, and Kannan~\cite{kannan}, who developed further geometric methods for integer programming.  For convex constraints, Hildebrand and K\"oppe~\cite{hk2010} give a Lenstra-type deterministic $2^{O(n\log n)}$ algorithm for quasiconvex polynomial integer optimization.  More recently, Reis and Rothvoss~\cite{reisrothvoss} gave a randomized $(\log n)^{O(n)}$-time algorithm, up to polynomial factors, for convex integer programming.

\BallIP{} can also be interpreted as a restricted integer trust-region problem; see Mor\'e and
Sorensen~\cite{moresorensen1983} for the continuous problem and
Del Pia~\cite{delpia2024} for the mixed-integer setting over general lattices. In particular, Del Pia~\cite{delpia2024} show that the integer trust region problem over general lattices is strongly NP-hard. We show that a even with
linear objective, an origin-centered Euclidean ball, and the standard integer
lattice up to scaling, the problem remains weakly NP-hard, that is, NP-hard
but solvable in pseudo-polynomial time.

\paragraph{Closest vector problem.}
The closest vector problem is one of the central computational problems in
lattice theory. The hardness of CVP goes back to van Emde Boas~\cite{vanemdeboas}; see also
Micciancio~\cite{micciancio2001}. Micciancio and
Voulgaris~\cite{miccianciovoulgaris} give a deterministic $2^{O(n)}$-time
algorithm for exact Euclidean CVP using Voronoi cells. A single-exponential-time,
polynomial-space algorithm remains an open question~\cite{hunkenschroder},
and fine-grained hardness is studied in~\cite{BGS17,AK22}.

\paragraph{Related reduction techniques.}
The binary-forcing step in Theorem~\ref{thm:ballip-reduction} is related to
Micciancio's reduction from \textsc{SubsetSum} to CVP~\cite{micciancio2001}.
Theorem~\ref{thm:main} combines bounded integer equations using powers of a
large base, treating the equations as digits of one integer. Such positional encodings appear in Karp~\cite{karp1972}; see
also the Diophantine equation aggregation literature~\cite{bradley1971,gloverwoolsey1972,kannan1983aggregation}.

\paragraph{Organization.}
Section~\ref{sec:notation} fixes notation, the normalizations used by the
algorithms, and the lattice terminology. Section~\ref{sec:knapsack-group}
proves NP-completeness for every fixed integer $p\ge2$.
Section~\ref{sec:exact} gives the two pseudo-polynomial exact algorithms for
the Euclidean case. Section~\ref{sec:cvp-geometry} develops the objective-level
geometry: the covering-radius algorithm, the discriminant obstruction, and the
reduction from CVP to \BallIP. Section~\ref{sec:boundary} treats proximity
bounds and boundary enumeration, Section~\ref{sec:ellipsoid-discussion}
discusses which constructions extend to ellipsoids, and
Section~\ref{sec:open} concludes with open questions.

\section{Preliminaries}\label{sec:notation}

\paragraph{Notation.}
Vectors are boldface, $[n]:=\{1,\ldots,n\}$, and $\boldsymbol{\chi}_S$ is the
indicator vector of $S\subseteq[n]$. We write $\vec e_i$ for the $i$th
standard unit vector and $I_k$ for the $k\times k$ identity matrix, $\norm{\cdot}_p$ for the
$\ell_p$ norm, $\norm{\cdot}$ for the Euclidean norm, and $\mathbb B^n$ for
the closed Euclidean unit ball. Set $\cmax:=\norm{\vec c}_\infty$ and
$\lfloor z\rceil:=\lfloor z+\tfrac12\rfloor$. The input length is
$L=\langle\vec c,r^2\rangle$, or $\langle\vec c,r^2,\gamma\rangle$ for a
decision instance; $\operatorname{poly}(L)$ denotes an unspecified polynomial
factor. The radius $r=\sqrt{r^2}$ need not be integral.

Write $\OPT(\vec c,r^2)$ for the value of \BallIP{} in
\eqref{eq:intro-ballip-signed} with $p=2$. A superscript $+$ on the problem
or its value restricts the variables to $\Z_{\ge0}^n$. By coordinatewise sign
symmetry,
\begin{equation}\label{eq:ballip}
    \OPT(\vec c,r^2)=\OPT(|\vec c|,r^2)=\OPT^+(|\vec c|,r^2).
\end{equation}
where $|\vec c|$ is obtained by taking absolute values coordinatewise. The same convention applies to $\BallIPp{p}$, with budget
$\rho=r^p\in\Z_{\ge0}$. We omit the subscript only when $p=2$.
Fixed-level subproblems retain unrestricted integer variables unless stated
otherwise.

\paragraph{Normalizations.}
For algorithms, we remove zero-cost coordinates, calling the remaining
coordinates \emph{active}, and normalize signs. If
$g=\gcd(c_1,\ldots,c_n)>0$, replacing $\vec c$ by $\vec c/g$ makes it
primitive; a decision threshold becomes $\lceil\gamma/g\rceil$.
An optimizer is recovered by restoring signs and zero coordinates, and its
value is multiplied by $g$. The cases $\vec c=\vec0$ and $r^2=0$ have
value zero; with one active coordinate, choose magnitude $\lfloor r\rfloor$.
Thus the algorithmic arguments may assume $n\ge2$, $r>0$, and positive
primitive costs. Reductions retain their explicitly constructed coefficients
and domains, with normalization at the end if needed.

\paragraph{Lattices and the closest vector problem.}
A \emph{lattice} $\Lambda=B\Z^k\subseteq\R^d$ is the set of integer
combinations of the columns of a matrix $B\in\R^{d\times k}$ with linearly
independent columns; $B$ is a \emph{basis}, $k$ is the \emph{rank}, $d$ is
the ambient dimension, and $\Lambda$ is \emph{full-rank} when $k=d$. For
$\vec t\in\R^d$, $\operatorname{dist}(\vec t,\Lambda)=\min_{\vec v\in\Lambda}\norm{\vec t-\vec v}$,
and the \emph{covering radius} of $\Lambda$ is
$\sup\{\operatorname{dist}(\vec t,\Lambda):\vec t\in\operatorname{span}\Lambda\}$.
For an integer $m$, the \emph{objective level} $m$ is the hyperplane
$\{\vec x\in\R^n:\vec c^\top\vec x=m\}$, its \emph{slice} is the intersection of
that hyperplane with $r\mathbb B^n$, and the level is \emph{feasible} if its
slice contains an integer point. A \emph{cap} of the ball is its intersection
with a closed halfspace $\{\vec x:\vec c^\top\vec x\ge\gamma\}$.

The \emph{closest vector problem} (CVP) comes in two forms that we use.
In the \emph{decision} form, the input is a basis $B\in\Z^{d\times k}$ of
rank $k$, a target $\vec t\in\Z^d$, and a bound $\beta\in\Z_{\ge0}$, and
the question is whether some $\vec z\in\Z^k$ satisfies
$\norm{B\vec z-\vec t}^2\le\beta$; this is the form that
Subsection~\ref{sec:cvp-to-ballip} reduces to \BallIP. In the \emph{search}
form, the output is a lattice vector $B\vec z$ minimizing
$\norm{B\vec z-\vec t}$; this is the form used as an oracle in
Subsection~\ref{sec:cvpalg}, where the lattice is fixed and only the target
varies from query to query, so a Voronoi cell of the lattice can be computed
once and reused~\cite{miccianciovoulgaris}. A rational target is handled by
scaling the lattice and the target by a common denominator, which does not
change which lattice vector is closest. The promise variant gap-CVP appears
only in Remark~\ref{rem:cvp-ballip-hardness}, where it is defined.

\section{Hardness}\label{sec:knapsack-group}

We first establish NP-completeness for the Euclidean problem and then extend
the same binary-forcing mechanism to every fixed integer exponent $p\ge2$.

\subsection{Euclidean NP-completeness and an integer trust-region corollary}\label{sec:reduction}

The source problem is a cardinality-restricted form of subset sum.

\begin{problem}[Exact $k$-item Subset Sum Problem, E-kSSP]\label{prob:ekssp}
Given $a_1,\dots,a_n\in\Z_{\ge0}$, a target $\beta\in\Z_{\ge0}$, and a
cardinality $k\in\{0,1,\dots,n\}$, determine whether there exists $S\subseteq[n]$
with $|S|=k$ and $a(S)=\beta$, where $a(S)=\sum_{i\in S}a_i$.
\end{problem}

E-kSSP is the cardinality-constrained subset-sum problem. For completeness, we give a short, self-contained reduction from \textsc{SubsetSum}, whose NP-completeness is standard~\cite{kpp2004}.

\begin{lemma}\label{lem:ekssp}
E-kSSP is NP-complete.
\end{lemma}

\begin{proof}
E-kSSP is in NP, with $S$ as a certificate. For hardness we reduce from
\textsc{SubsetSum}: given positive integers $s_1,\dots,s_m$ and a target $t$,
decide whether some $S'\subseteq[m]$ has $\sum_{i\in S'}s_i=t$. Form the E-kSSP
instance with $n=2m$ items, where $a_i=s_i$ for $1\le i\le m$, $a_i=0$
for $m<i\le 2m$, $\beta=t$, and $k=m$.
If $S'\subseteq[m]$ satisfies $\sum_{i\in S'}s_i=t$, then since $|S'|\le m$ and
the items $m+1,\dots,2m$ have value zero, adjoining $m-|S'|$ of them to $S'$ yields a
set $S$ with $|S|=m=k$ and $a(S)=t=\beta$. Conversely, if $|S|=k$ and
$a(S)=\beta$, then $S':=S\cap[m]$ satisfies $\sum_{i\in S'}s_i=a(S')=a(S)=t$,
since the indices in $S\setminus[m]$ carry zero items. The construction is
polynomial, so the instances are equivalent.
\end{proof}

We now establish the main hardness result for the Euclidean problem.

\begin{theorem}\label{thm:ballip-reduction}
The decision versions of \BallIP$^+$ and \BallIP{} are NP-complete.
\end{theorem}

\begin{proof}
By the sign-symmetry normalization in Section~\ref{sec:notation}, it suffices
for NP-hardness to work with positive costs and the nonnegative formulation;
thus, without loss of generality for this proof, we assume
$\vec x\in\Z_{\ge0}^n$ and prove hardness for \BallIP$^+$.
Both formulations are in NP. Indeed, if $\vec x$ is feasible, then
$x_i^2\le r^2$ for every $i$, so each coordinate of $\vec x$ has encoding
length polynomial in the encoding length of $r^2$; the ball constraint and the
objective threshold can then be checked in polynomial time.

For NP-hardness, we reduce from E-kSSP, which is NP-complete by Lemma~\ref{lem:ekssp}. Given an instance
$(a_1,\dots,a_n,\beta,k)$, set
\[
  M=n+1,\qquad
  u_i=Ma_i,\qquad
  c_i=2u_i+1=2Ma_i+1,
\]
and let
\[
  B=2M\beta+k,\qquad
  r^2=\norm{\vec u}^2+B,\qquad
  \gamma=\vec c^\top\vec u+B.
\]
All constructed data have encoding length polynomial in the E-kSSP input.

Write $\vec x=\vec u+\vec d$. Because $\vec u\in\Z^n_{\ge0}$, the condition
$\vec x\in\Z^n_{\ge0}$ is equivalent to $\vec d\in\Z^n$ with
$d_i\ge-u_i$ for every $i$. Since $2u_i=c_i-1$, the ball constraint is equivalent to
$\sum_i(c_i-1)d_i+\sum_i d_i^2\le B$. Define
$F(\vec d)=\sum_i(c_i-1)d_i+\sum_i d_i^2$. Likewise, the objective threshold
is equivalent to $\sum_i c_i d_i\ge B$. Define $L(\vec d)=\sum_i c_i d_i$. Hence the constructed \BallIP$^+$ instance is a yes-instance exactly when there is
some $\vec d\in\Z^n$ with $d_i\ge-u_i$ satisfying
\[
  F(\vec d)\le B\le L(\vec d).
\]

For every $\vec d\in\Z^n$,
\[
  F(\vec d)-L(\vec d)
  =\sum_i\bigl(d_i^2-d_i\bigr)
  =\sum_i d_i(d_i-1)\ge0.
\]
Each $d_i(d_i-1)$ is the product of two consecutive integers, so it is
nonnegative and vanishes exactly when $d_i\in\{0,1\}$. Therefore any feasible
$\vec d$ satisfies $B\le L(\vec d)\le F(\vec d)\le B$. It follows that
$F(\vec d)=L(\vec d)=B$ and $\vec d\in\{0,1\}^n$. Thus $\vec d=\boldsymbol{\chi}_S$ for some $S\subseteq[n]$.
Figure~\ref{fig:binary-forcing} illustrates geometrically the binary-forcing step above, in which feasibility restricts each displacement $d_i$ to $0$ or $1$.

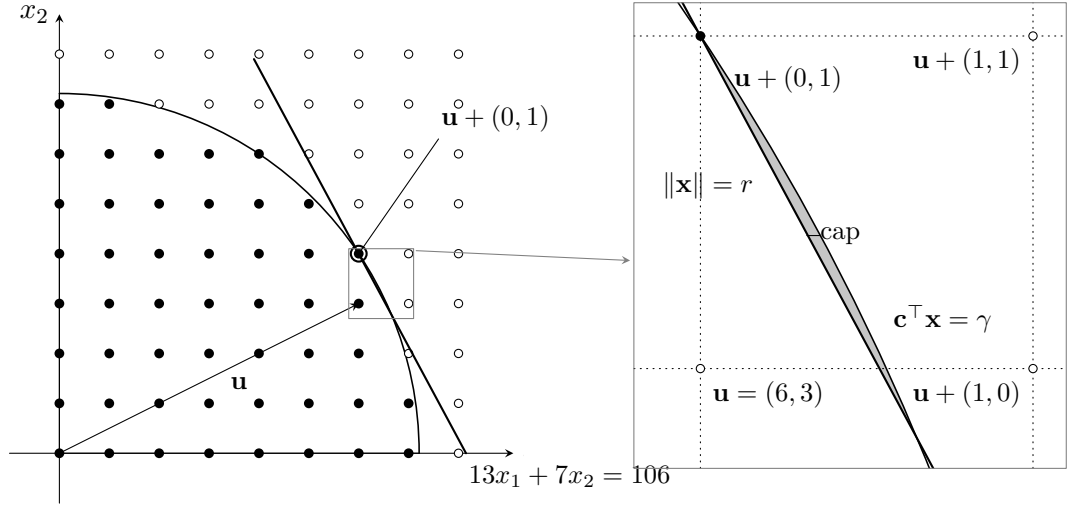
\begin{figure}[t]
\centering
\begin{tikzpicture}[>=stealth, line cap=round, line join=round]
  \tikzset{
    inpt/.style  = {circle, fill, inner sep=1.3pt},
    outpt/.style = {circle, draw, fill=white, inner sep=1.1pt},
    ptlabel/.style = {font=\small}
  }
  \pgfmathsetmacro{\rad}{sqrt(52)}
  \begin{scope}[scale=0.66]
    \begin{scope}
      \clip (-1.2,-1.2) rectangle (9.3,9.0);
      \begin{scope}
        \clip (0,0) circle (\rad);
        \fill[gray!45] (5.5,4.929) -- (7.7,0.843) -- (9,9) -- (5.5,9) -- cycle;
      \end{scope}
      \draw[->,thin] (-1,0) -- (9.1,0) node[right] {$x_1$};
      \draw[->,thin] (0,-1) -- (0,8.8) node[left] {$x_2$};
      \draw[semithick] (\rad,0) arc (0:90:\rad);
      \draw[thin] (0,0) -- (\rad,0); \draw[thin] (0,0) -- (0,\rad);
      \foreach \i in {0,...,8}{\foreach \j in {0,...,8}{
        \pgfmathparse{(\i*\i+\j*\j<=52)?1:0}
        \ifnum\pgfmathresult=1 \node[inpt] at (\i,\j) {}; \else \node[outpt] at (\i,\j) {}; \fi
      }}
      \draw[thick] (3.9,7.9) -- (8.15,0.0);
      \draw[->] (0,0) -- (6,3) node[pos=0.55, below right=-2pt, ptlabel] {$\vec u$};
      \draw[thick] (6,4) circle (4.5pt);
      \draw[gray, thin] (5.8,2.7) rectangle (7.1,4.1);
    \end{scope}
    \node[ptlabel, anchor=north west] at (8.0,-0.05) {$13x_1+7x_2=106$};
    \draw[thin] (6,4) -- (7.6,6.3);
    \node[ptlabel, anchor=south west] at (7.45,6.25) {$\vec u+(0,1)$};
  \end{scope}
  \begin{scope}[xshift=7.6cm, yshift=-0.2cm, scale=4.4, shift={(-5.8,-2.7)}]
    \begin{scope}
      \clip (5.8,2.7) rectangle (7.1,4.1);
      \begin{scope}
        \clip (0,0) circle (\rad);
        \fill[gray!45] (5.0,5.857) -- (8.0,0.286) -- (9,9) -- (5,9) -- cycle;
      \end{scope}
      \draw[semithick] (0,0) circle (\rad);
      \draw[thick] (5.0,5.857) -- (8.0,0.286);
      \draw[dotted] (5.8,3) -- (7.1,3); \draw[dotted] (5.8,4) -- (7.1,4);
      \draw[dotted] (6,2.7) -- (6,4.1); \draw[dotted] (7,2.7) -- (7,4.1);
      \node[outpt] at (6,3) {}; \node[outpt] at (7,3) {};
      \node[inpt]  at (6,4) {}; \node[outpt] at (7,4) {};
    \end{scope}
    \draw[gray, thin] (5.8,2.7) rectangle (7.1,4.1);
    \node[ptlabel, below right=1pt] at (6,3) {$\vec u=(6,3)$};
    \node[ptlabel, below left=1pt] at (7,3) {$\vec u+(1,0)$};
    \node[ptlabel, anchor=north west] at (6.07,3.94) {$\vec u+(0,1)$};
    \node[ptlabel, below left=1pt] at (7,4) {$\vec u+(1,1)$};
    \node[ptlabel, anchor=east] at (6.19,3.55) {$\norm{\vec x}=r$};
    \node[ptlabel, anchor=west] at (6.55,3.15) {$\vec c^\top\vec x=\gamma$};
    \node[ptlabel, anchor=west] at (6.33,3.40) {cap};
    \draw[thin] (6.32,3.40) -- (6.36,3.40);
  \end{scope}
  \draw[gray, thin, ->] (4.72,2.68) -- (7.55,2.55);
\end{tikzpicture}
\caption{Binary forcing for $\vec u=(6,3)$, $\vec c=(13,7)$, $r^2=52$, and $\gamma=106$. Left is the constructed instance. The ball and the objective threshold leave only the shaded cap as the feasible region. Right panel magnifies around the lattice points of $u$. Of these, only $\vec u+(0,1)$ lies in the cap. For any $\vec d = \vec x-\vec u$, $\vec d \in \{0,1\}$ is a necessary but not sufficient condition.}
\label{fig:binary-forcing}
\end{figure}

Using the definition of $\vec c$, we have
$B=L(\vec d)=\sum_{i\in S}c_i=2M a(S)+|S|$. Since
$B=2M\beta+k$, we obtain $2M\bigl(a(S)-\beta\bigr)=k-|S|$.

The right-hand side has absolute value at most $n$, whereas the left-hand side
is an integer multiple of $2M$. Because $M=n+1$, we have $n<2M$, so both sides
must be zero. Hence $a(S)=\beta$ and $|S|=k$. Thus every feasible point of the constructed \BallIP$^+$ instance yields an E-kSSP
witness. 

Conversely, suppose $S\subseteq[n]$ is an E-kSSP witness, so that
$a(S)=\beta$ and $|S|=k$. Take $\vec d=\boldsymbol{\chi}_S$ and
$\vec x=\vec u+\vec d$. Then $\vec x\in\Z^n_{\ge0}$ and
$L(\vec d)=2Ma(S)+|S|=2M\beta+k=B$. Since $\vec d$ is binary, $d_i^2=d_i$ for every $i$, and therefore
$F(\vec d)=L(\vec d)=B$. It follows that $\norm{\vec x}^2\le r^2$ and
$\vec c^\top\vec x\ge\gamma$. Hence the \BallIP$^+$ instance is feasible.

The E-kSSP instance is therefore a yes-instance if and only if the constructed
\BallIP$^+$ instance is a yes-instance. Since the construction is polynomial,
\BallIP$^+$ is NP-hard. The normalization in
Section~\ref{sec:notation} transfers the same positive-cost instances to
\BallIP, so both decision problems are NP-complete.
\end{proof}

We next show that the hardness remains even when $r$ itself is an integer, or equivalently, when $r^2$ is a perfect square. We use this to show that the integer trust-region problem is NP-hard even when the lattice is restricted to a scalar multiple of $\mathbb Z^n$.

\begin{lemma}[Integral-radius version of the Euclidean reduction]\label{lem:ballip-integer-radius}
The reduction in Theorem~\ref{thm:ballip-reduction} may be chosen so that the
radius $r$ is a positive integer. Let $m$ denote the dimension of the
constructed instance. The same constructed instances are equivalent
to the E-kSSP instance whether the variables range over $\Z^m_{\ge0}$ or over
all of $\Z^m$.
\end{lemma}

\begin{proof}
Start with an E-kSSP instance. If its cardinality target $k$ is even,
adjoin one item $A=a([n])+1$ and replace $(\beta,k)$ by
$(\beta+A,k+1)$. Any witness for the new instance must contain the new
item, so this preserves the witnesses and makes $k$ odd. Next replace
$a_i$ by $2(a_i+1)$ and $\beta$ by $2(\beta+k)$. Since every admissible
set has exactly $k$ elements, this transformation also preserves the
witnesses. Thus we may assume that all items are positive and even and that
$k$ is odd.

Apply Theorem~\ref{thm:ballip-reduction} to the resulting instance and write
$R=\norm{\vec u}^2+B$. Each $u_i=Ma_i$ is even, whereas
$B=2M\beta+k$ is odd, so $R$ is odd. Put
\[
  s=\frac{R-1}{2},\qquad u_0=s,\qquad c_0=2s+1=R,
\]
and adjoin this coordinate to $\vec u$ and $\vec c$, leaving $B$ unchanged.
For $\vec u'=(u_0,\vec u)$ and $\vec c'=(c_0,\vec c)$, set
\[
  (r')^2=\norm{\vec u'}^2+B,
  \qquad
  \gamma'=(\vec c')^\top\vec u'+B.
\]
Then $(r')^2=s^2+R=(s+1)^2$, so $r'=s+1$ is an integer. Moreover,
$c_i'=2u_i'+1$ in every coordinate, and hence for
$\vec x'=\vec u'+\vec d'$ the same calculation as in
Theorem~\ref{thm:ballip-reduction} gives
\[
  F'(\vec d')-L'(\vec d')=\sum_i d_i'(d_i'-1)\ge0,
\]
with equality exactly for binary $\vec d'$. Since every preprocessed item is
positive, $c_0=R>B$, so a binary vector satisfying $L'(\vec d')=B$ must have
$d_0'=0$. Its remaining coordinates therefore encode exactly an E-kSSP
witness. Finally, the squeeze
$B\le L'(\vec d')\le F'(\vec d')\le B$ uses only the displayed inequality,
which holds for every integer $\vec d'$. Thus the construction is equivalent
over both $\Z^m_{\ge0}$ and $\Z^m$, and all data have polynomial encoding
length.
\end{proof}

Del Pia~\cite[Section~3, Theorem~3]{delpia2024} proves strong NP-hardness and
constant-accuracy inapproximability for the general mixed-integer trust-region
problem using a Max-Cut reduction, and separately proves feasibility hardness
using CVP. Here the approximation notion is normalized by the attainable
objective range. Corollary~\ref{thm:tr} shows that exact NP-hardness already
holds with a scalar multiple of the standard lattice, no shift, an
origin-centered ball, and a linear objective. The hardness of this restricted case is weak, as shown by the
pseudo-polynomial algorithms in Section~\ref{sec:exact}. 

\begin{corollary}[Integer trust region with the simplest lattice]\label{thm:tr}
The integer trust-region problem
\[
  \min\{\,\vec x^\top H\vec x+\vec h^\top\vec x : \norm{\vec x}\le1,\ \vec
  x\in\Lambda\,\}
\]
where $H\in\Q^{n\times n}$ is symmetric and $\vec h\in\Q^n$, is NP-hard already when $H=0$ and $\Lambda=\tfrac1r\Z^n$ for a positive integer $r$, a scalar multiple of
the standard lattice with no shift and with rational generators.
\end{corollary}

\begin{proof}
By Lemma~\ref{lem:ballip-integer-radius}, the hard Euclidean
\BallIP{} instances may be chosen with $r\in\Z_{>0}$ and remain hard when
the variables range over all of $\Z^n$. For such an instance, set
$\Lambda=\tfrac1r\Z^n$, $H=0$, and $\vec h=-r\vec c$. Every
$\vec y\in\Lambda$ has the form $\vec y=\vec x/r$ with $\vec x\in\Z^n$, and
then $\norm{\vec y}\le1$ if and only if $\norm{\vec x}\le r$, while
$\vec h^\top\vec y=-\vec c^\top\vec x$. Thus the trust-region optimum is at
most $-\gamma$ exactly when the corresponding \BallIP{} instance has
optimum at least $\gamma$.
\end{proof}


\subsection{NP-completeness for fixed integer exponents $p\ge2$}\label{sec:hardness}\label{sec:lp}

Fix an integer $p\ge2$. The input consists of an integer vector $\vec c$, a
nonnegative integer $\rho$, and an integer $\gamma$, and the constraint is
$\sum_i |x_i|^p\le\rho$ in the unrestricted formulation (equivalently
$\sum_i x_i^p\le\rho$ in the nonnegative formulation). The decision question
asks whether there is a vector in the indicated domain satisfying this budget
and $\vec c^\top\vec x\ge\gamma$. Here $\rho$ encodes the $p$-th power of the
radius, so that all data are integers. By the same
coordinatewise sign symmetry used in Section~\ref{sec:notation}, replacing
$\vec c$ by $|\vec c|$ lets us, without loss of generality for the hardness
proof below, assume $\vec x\in\Z_{\ge0}^n$ and work with
$\BallIPp{p}^{+}$. 

Next, we extend the proof for the Euclidean case to every fixed integer exponent
$p\ge2$, to obtain the following family-wide hardness result.

\begin{theorem}\label{thm:lp}
For every fixed integer $p\ge2$, the decision versions of
$\BallIPp{p}^{+}$ and $\BallIPp{p}$ are NP-complete.
\end{theorem}

Our strategy will be again a reduction from E-kSSP, where the input is
$a_1,\ldots,a_n\in\Z_{\ge0}$, a target $\beta\in\Z_{\ge0}$, and a cardinality
$k\in\{0,\ldots,n\}$, and a witness is a set $S\subseteq[n]$ satisfying
$a(S)=\beta$ and $|S|=k$. The proof follows the same structure as the $p=2$ case.
We write $\vec x=\vec u+\vec d$, so that $d_i$ records how much the
coordinate $u_i$ is changed to obtain $x_i$. We choose the construction so
that every feasible solution has $d_i\in\{0,1\}$ for all $i$; then
$\vec d=\boldsymbol{\chi}_S$ records the selected set $S$. The choice of $U$ will also
ensure that feasibility forces $|S|=k$ and $a(S)=\beta$.
Geometrically, the same picture as Figure~\ref{fig:binary-forcing} applies after replacing the Euclidean ball by the corresponding $\ell_p$ ball.

\subsubsection{Forcing $\vec d$ to be binary}

For $t\ge0$, set $q(t)=(t+1)^p-t^p$. Here, $q(u)$ is the marginal increase in the $p$-th-power budget when
an integer coordinate is increased from $u$ to $u+1$. For an integer $d\ge-u$, define
$R_u(d):=(u+d)^p-u^p-q(u)d$. 

\begin{lemma}\label{lem:tangent}
For every integer $u\ge1$ and every integer $d\ge-u$, $R_u(d)\ge0$, with
equality exactly when $d\in\{0,1\}$. Moreover, if $d\notin\{0,1\}$, then
$R_u(d)\ge q(u)-q(u-1)$.
\end{lemma}

\begin{proof}
The function $q(t)=(t+1)^p-t^p$ is strictly increasing for $t\ge0$.
The increments $q(t+1)-q(t)$ are nondecreasing in $t$, because
\[
q''(t)=p(p-1)\bigl((t+1)^{p-2}-t^{p-2}\bigr)\ge 0.
\]

If $d\ge1$, telescoping gives $(u+d)^p-u^p=\sum_{j=0}^{d-1}q(u+j)$, and
hence $R_u(d)=\sum_{j=1}^{d-1}\bigl(q(u+j)-q(u)\bigr)$. This is $0$ for
$d=1$. If $d\ge2$, every term is positive, and the first one satisfies
$q(u+1)-q(u)\ge q(u)-q(u-1)$. For $d=0$, $R_u(d)=0$ again.

Finally, let $d=-h$ with $1\le h\le u$. Telescoping downward gives
$R_u(-h)=\sum_{j=1}^{h}\bigl(q(u)-q(u-j)\bigr)$, a sum of
positive terms whose first term is $q(u)-q(u-1)$.
\end{proof}

The reason for choosing the objective coefficients from the increments
$q(u)$ is the following identity. For arbitrary integer base points
$u_i\ge1$, put $c_i=q(u_i)$. Then, for every integer vector $\vec d$ with $d_i\ge-u_i$,
\begin{equation}\label{eq:increment-identity}
\sum_i (u_i+d_i)^p-\sum_i u_i^p
=
\sum_i c_i d_i
+
\sum_i R_{u_i}(d_i).
\end{equation}
Thus the change in the norm budget equals the change in the linear objective
plus $\sum_i R_{u_i}(d_i)$. By Lemma~\ref{lem:tangent}, every term in
this sum is nonnegative, and all of them vanish exactly when
$d_i\in\{0,1\}$ for every $i$. For $p=2$, $q(t)=2t+1$ and
$R_u(d)=d(d-1)$, so \eqref{eq:increment-identity} is exactly the quadratic
identity used in the Euclidean version to force $d_i\in\{0,1\}$ for every $i$.

\subsubsection{Choosing $U$}
For the $p=2$ case, the quantity $B=2M\beta+k$ can be viewed as a
positional encoding in base $2M$ with $k$ as its lower digit. Since this digit is smaller than $2M$, the
equality $2M\beta+k=2Ma(S)+|S|$ forces $a(S)=\beta$ and $|S|=k$. The construction for general $p$ uses the same
idea, but there are now three different terms that must be kept from
interfering with one another. 

Define $T=\max\{\sum_i a_i,\beta\}$. For a shift $U>\max_i a_i$, write
$e_i=q(U+a_i)-q(U)-q'(U)a_i$ and $E=\sum_i e_i$. Thus $e_i$ is the difference between the actual increase
$q(U+a_i)-q(U)$ and the term $q'(U)a_i$. Convexity gives
$e_i\ge0$.

We choose $U$ so that feasibility of the constructed $\BallIPp{p}^{+}$
instance forces $d_i\in\{0,1\}$ for every $i$ and then forces the set
$S=\{i:d_i=1\}$ to be an E-kSSP solution. The inequality
$q(U)-q(U-1)>E$ ensures that if even one coordinate satisfies
$d_i\notin\{0,1\}$, its term $R_{u_i}(d_i)$ already exceeds the entire
allowance $E$. Once every $d_i$ is binary, $q'(U)>E$ ensures that changing
the subset sum by even one cannot be compensated by the terms $e_i$.
Finally, $q(U)>q'(U)T+E$ ensures that changing the cardinality by one cannot
be compensated by any possible change in the subset sum or by the terms
$e_i$.

The next lemma shows that all three inequalities can be enforced
simultaneously while keeping the encoding length of $U$ polynomial.

\begin{lemma}\label{lem:scale}
One can compute in polynomial time an integer $U>\max_i a_i$, of polynomial
encoding length, such that
\begin{equation}\label{eq:three-scales}
   q(U)>q'(U)T+E,\qquad
   q'(U)>E,\qquad
   q(U)-q(U-1)>E .
\end{equation}
\end{lemma}

\begin{proof}
Set $A_2=\sum_i a_i^2$. For fixed $p$ and large $U$, we have
$q(U)=\Theta_p(U^{p-1})$, $q'(U)=\Theta_p(U^{p-2})$,
$q(U)-q(U-1)=\Theta_p(U^{p-2})$, and
$E=O_p(U^{p-3}A_2)$. Thus $q(U)$ grows one power of $U$ faster than
$q'(U)$, while both $q'(U)$ and $q(U)-q(U-1)$ grow one power of $U$ faster
than $E$. Consequently, once $U$ is sufficiently large compared with the
E-kSSP data, the three inequalities in \eqref{eq:three-scales} hold.

If $p=2$, take $U=1+\max\{T,\max_i a_i\}$. Then $E=0$, $q'(U)=2$,
$q(U)-q(U-1)=2$, and $q(U)=2U+1>2T$.

Now suppose $p\ge3$ and take
\[
 U=
 1+\max\left\{
     \max_i a_i,\,
     2^pT,\,
     2^{2p-5}(p-1)A_2
 \right\}.
\]
The binomial expansion gives $q(U)\ge pU^{p-1}$,
$p(p-1)U^{p-2}\le q'(U)\le p2^{p-1}U^{p-2}$, and
$q(U)-q(U-1)=(U+1)^p-2U^p+(U-1)^p\ge p(p-1)U^{p-2}$.
Also, Taylor's formula with integral remainder gives
\[
0\le e_i
\le
\frac12\max_{[U,2U]}q''\,a_i^2
\le
p(p-1)2^{2p-6}U^{p-3}a_i^2,
\]
and therefore
$E\le p(p-1)2^{2p-6}U^{p-3}A_2$. The chosen value of $U$ implies
$q'(U)T<\frac p2U^{p-1}$ and $E<\frac p2U^{p-2}$, from which all three
inequalities in \eqref{eq:three-scales} follow. Since $p$ is fixed, the
displayed value of $U$ has polynomial encoding length and is computable in
polynomial time.
\end{proof}

\subsubsection{The reduction}

\begin{proof}[Proof of Theorem~\ref{thm:lp}]
The problem belongs to NP. Any feasible point satisfies
$0\le x_i^p\le\rho$, hence $0\le x_i\le\rho$, so a feasible point has
polynomial encoding length and the two defining inequalities can be checked
in polynomial time.

For NP-hardness, choose $U$ by Lemma~\ref{lem:scale} and define
$u_i=U+a_i$ and $c_i=q(u_i)>0$. By the definition of $e_i$,
\begin{equation}\label{eq:ci-three-parts}
        c_i=q(U)+q'(U)a_i+e_i.
\end{equation}
Thus $c_i$ is the sum of three terms: $q(U)$, which records that one more
item was selected; $q'(U)a_i$, which records the value $a_i$ of that item;
and $e_i$.

Set $M=q(U)k+q'(U)\beta$ and define the
$\BallIPp{p}^{+}$ instance by
\[
    \rho=\sum_i u_i^p+M+E,
    \qquad
    \gamma=\vec c^\top\vec u+M.
\]

Suppose first that $\vec x=\vec u+\vec d$ is feasible. The objective
constraint gives
\begin{equation}\label{eq:objective-increment}
        \sum_i c_i d_i\ge M.
\end{equation}
Using \eqref{eq:increment-identity}, the norm constraint gives
$\sum_i c_i d_i+\sum_iR_{u_i}(d_i)\le M+E$. Combining the two inequalities,
\begin{equation}\label{eq:R-budget}
        \sum_iR_{u_i}(d_i)\le E.
\end{equation}

We first show that every $d_i$ must be either $0$ or $1$. If some
$d_j\notin\{0,1\}$, Lemma~\ref{lem:tangent} gives
$R_{u_j}(d_j)\ge q(u_j)-q(u_j-1)$. Since $u_j=U+a_j\ge U$ and
$q(t)-q(t-1)$ is nondecreasing in $t$,
\[
R_{u_j}(d_j)
\ge q(u_j)-q(u_j-1)
\ge q(U)-q(U-1)
>E.
\]
This contradicts \eqref{eq:R-budget}. Therefore
$\vec d=\boldsymbol{\chi}_S$ for some $S\subseteq[n]$.

For such a vector $\vec d=\boldsymbol{\chi}_S$, every $R_{u_i}(d_i)$ is zero. Feasibility
therefore gives $M\le\sum_{i\in S}c_i\le M+E$. Using
\eqref{eq:ci-three-parts}, let
$\Delta_k=|S|-k$, $\Delta_a=a(S)-\beta$, and
$\varepsilon=\sum_{i\in S}e_i$. Since $0\le\varepsilon\le E$ and
$|\Delta_a|\le T$, we obtain
\begin{equation}\label{eq:two-scale-decode}
        0
        \le
        q(U)\Delta_k+q'(U)\Delta_a+\varepsilon
        \le
        E.
\end{equation}

We next show that $S$ has the required cardinality. If $\Delta_k\ge1$, then
$q(U)\Delta_k+q'(U)\Delta_a+\varepsilon\ge q(U)-q'(U)T>E$, contradicting the upper
bound in \eqref{eq:two-scale-decode}. If $\Delta_k\le-1$, then
$q(U)\Delta_k+q'(U)\Delta_a+\varepsilon\le -q(U)+q'(U)T+E<0$, contradicting the lower
bound. Hence $\Delta_k=0$, so $|S|=k$.

It remains to show that the selected items have sum $\beta$. With $\Delta_k=0$,
\eqref{eq:two-scale-decode} becomes
$0\le q'(U)\Delta_a+\varepsilon\le E$. If $\Delta_a\ge1$, then
$q'(U)\Delta_a+\varepsilon\ge q'(U)>E$, while if $\Delta_a\le-1$, then
$q'(U)\Delta_a+\varepsilon\le -q'(U)+E<0$. Both are impossible, so $\Delta_a=0$ and
therefore $a(S)=\beta$. Thus every feasible $\BallIPp{p}^{+}$ point
determines an E-kSSP witness.

Conversely, suppose $S$ is an E-kSSP witness. Take $\vec d=\boldsymbol{\chi}_S$ and
$\vec x=\vec u+\vec d$. Since $\vec d=\boldsymbol{\chi}_S$, every
$d_i\in\{0,1\}$, so all terms $R_{u_i}(d_i)$ vanish. Moreover,
\[
\sum_{i\in S}c_i
=
q(U)k+q'(U)\beta+\sum_{i\in S}e_i
=
M+\varepsilon
\]
for some $0\le\varepsilon\le E$. Hence the objective increment is at least
$M$, while by \eqref{eq:increment-identity} the norm increment is
$M+\varepsilon\le M+E$. Thus $\vec x$ is feasible.

All constructed quantities are integers of polynomial encoding length and
are computable in polynomial time for fixed $p$. Therefore
$\BallIPp{p}^{+}$ is NP-hard, and together with membership in NP it is
NP-complete. The sign-symmetry normalization from Section~\ref{sec:notation}
then gives NP-completeness of the unrestricted $\BallIPp{p}$ formulation
as well.
\end{proof}

\begin{remark}\label{rem:endpoint-norms}\label{prop:endpoint-norms}
    For completeness, we record the endpoint cases. For $p=1$ and $p=\infty$, both problems are trivially solvable in polynomial time. 
\end{remark}


\section{Exact algorithms}\label{sec:exact}

We now return to the Euclidean problem \BallIP. We give two
pseudo-polynomial exact algorithms. We use the preprocessing from
Section~\ref{sec:notation}, so throughout this section the active costs are
positive and primitive, $n\ge2$, and $r>0$; the variables may still be
unrestricted when a fixed objective level is considered. The Knapsack DP is
polynomial in the numerical value of the squared-radius budget $r^2$, whereas
the second dynamic program is polynomial in the binary encoding length of
$r^2$ but pseudo-polynomial in the cost magnitudes.

\subsection{Knapsack DP: polynomial in the radius magnitude}\label{sec:dp}

The ball constraint can be read as a knapsack of capacity $r^2$: assigning a
nonnegative value $m$ to coordinate $i$ uses $m^2$ units of capacity and earns
$c_i m$ units of profit. Hochbaum~\cite[Section~1]{hochbaum1995} gives a dynamic program
for nonlinear knapsack problems that applies directly here. We state the
recurrence for completeness. Its pseudo-polynomial running time, together with
the hardness result, shows that the problem is weakly
NP-hard. Throughout this subsection we use the nonnegative formulation
\BallIP$^+$.

\paragraph{The recurrence.}
For $j\in\{0,1,\dots,n\}$ and $b\in\{0,1,\dots,r^2\}$ define
\[
  P(j,b)=\max\Bigl\{\textstyle\sum_{i=1}^j c_i x_i :
  \sum_{i=1}^j x_i^2\le b,\ x_i\in\Z_{\ge0}\ \text{for all }i\Bigr\}.
\]
The base case is $P(0,b)=0$ for all $b$, and the optimal value is
$P(n,r^2)$. Processing coordinates one at a time, for $j\ge1$ and $b\ge0$,
\begin{equation}\label{eq:dprec}
  P(j,b)=\max_{0\le x_j\le\lfloor\sqrt b\rfloor}
  \Bigl[\,c_j\,x_j+P\bigl(j-1,\,b-x_j^2\bigr)\,\Bigr].
\end{equation}

An optimal solution is recovered by storing a predecessor choice for each state and backtracking from $P(n,r^2)$.

\begin{proposition}[Running time]\label{prop:dp}
The Knapsack DP \eqref{eq:dprec} is exact and uses
$O\!\left(n(r^2)^{3/2}\right)=O(nr^3)$ arithmetic operations on integers of
polynomial bit length. Consequently its bit running time is
$O(nr^3)\operatorname{poly}(L)$, where $L=\langle\vec c,r^2\rangle$.
\end{proposition}

\begin{proof}
The outer loop runs $n$ times. For each coordinate $j$, the table is filled for
all $b$ from $0$ to $r^2$, and for each $b$ the variable $x_j$ ranges over at
most $\lfloor\sqrt b\rfloor+1$ values. The total work for one coordinate is thus
\[
  \sum_{b=0}^{r^2}\bigl(\lfloor\sqrt b\rfloor+1\bigr)=O\bigl((r^2)^{3/2}\bigr),
\]
and multiplying by $n$ coordinates gives $O\bigl(n\,(r^2)^{3/2}\bigr)$
arithmetic operations. Every table value is an objective value of a feasible
partial solution and therefore has bit length polynomial in $L$; the same is
true of the stored predecessor indices. This gives the stated bit-complexity
bound.
\end{proof}

\begin{remark}
Hochbaum's nonlinear-knapsack dynamic program extends directly to
\BallIPp{p} for every fixed finite integer $p\ge2$. Together with
Theorem~\ref{thm:lp}, this makes the decision version weakly NP-complete.
\end{remark}

\subsection{Dynamic programming over candidate objective values}\label{sec:window}

The optimal value lies among at most
$\norm{\vec c}_1+1\le n\cmax+1$ integer objective values below the continuous
optimum $r\norm{\vec c}$. Choose an index $h$ for which $c_h$ is the smallest active cost and set
\[
  D_h:=\frac12\sqrt{\norm{\vec c}^2+(n-2)c_h^2}
  \le \cmax\sqrt{\frac{n-1}{2}}.
\]
For each candidate value $V$, it suffices to search an $\ell_\infty$-box of
radius $\lceil D_h\rceil+1$ around the continuous minimum-norm point
$(V/\norm{\vec c}^2)\vec c$. Thus neither the number of candidate levels nor
the width of the search box depends on the numerical value of $r$.

As fixed in Section~\ref{sec:notation}, \BallIP{} denotes the unrestricted
formulation, so the variables in the fixed-level subproblems are not constrained
to be nonnegative. The preprocessing at the start of this section nevertheless
allows us to assume that all active costs $c_i$ are positive and that $\vec c$
is primitive.

\begin{lemma}[Bounds on the optimal objective value]\label{lem:window}
The optimal value satisfies
\[
  r\norm{\vec c}-\norm{\vec c}_1
  \le \OPT
  \le r\norm{\vec c}.
\]
Hence $\OPT$ is an integer in
\[
  \left[
    \max\!\left\{0,\left\lceil r\norm{\vec c}-\norm{\vec c}_1\right\rceil\right\},
    \left\lfloor r\norm{\vec c}\right\rfloor
  \right],
\]
which contains at most $\norm{\vec c}_1+1\le n\cmax+1$ integers.
\end{lemma}

\begin{proof}
For every feasible $\vec x$, Cauchy--Schwarz gives
$\vec c^\top\vec x\le\norm{\vec c}\norm{\vec x}\le r\norm{\vec c}$.
For the lower bound, set
$\bar x_i=\lfloor r c_i/\norm{\vec c}\rfloor$. Then
$\norm{\bar{\vec x}}^2\le r^2$ and
$\vec c^\top\bar{\vec x}\ge r\norm{\vec c}-\norm{\vec c}_1$.
Finally, $\OPT\ge0$ because $\vec0$ is feasible.
\end{proof}

We scan these candidate values from largest to smallest. Fix one such value
$V$. The unique minimum-norm real point on the hyperplane
$\vec c^\top\vec x=V$ is
$\xc=(V/\norm{\vec c}^2)\vec c$.

The following lemma gives the distance bound used by the second dynamic
program.

\begin{lemma}[Closeness]\label{lem:close}
For every $V\in\Z$ and every real point $\vec x^c$ satisfying
$\vec c^\top\vec x^c=V$, there is an integer point $\hat{\vec x}$ satisfying
\[
  \vec c^\top\hat{\vec x}=V,
  \qquad
  \norm{\hat{\vec x}-\vec x^c}\le D_h.
\]
In particular, this holds for the continuous center
$\xc=(V/\norm{\vec c}^2)\vec c$, and a closest integer point on that objective
level also lies within distance $D_h$ of $\xc$.
\end{lemma}

\begin{proof}
Since $\gcd_i c_i=1$, there is $\vec x_0\in\Z^n$ with
$\vec c^\top\vec x_0=V$ by B\'ezout's identity. For each $i\ne h$, let
\[
  \vec w_i=c_h\vec e_i-c_i\vec e_h.
\]
These are integer kernel vectors. They are linearly independent. Indeed, if $
  \sum_{i\ne h}\lambda_i\vec w_i=0,
$ then coordinate $i\ne h$ equals $\lambda_i c_h$, forcing every
$\lambda_i=0$. Since there are $n-1$ of them, they form a real basis of
$H_0=\{\vec y\in\R^n:\vec c^\top\vec y=0\}$. They need not form a
lattice basis.

Write
\[
  \vec x^c-\vec x_0=\sum_{i\ne h}\alpha_i\vec w_i.
\]
For each $i\ne h$, write $\alpha_i=m_i+\theta_i$ with $m_i\in\Z$ and
$0\le\theta_i<1$, and independently choose an integer $A_i$ equal to $m_i+1$
with probability $\theta_i$ and to $m_i$ otherwise. Then
\[
  \mathbb E(A_i-\alpha_i)=0,
  \qquad
  \mathbb E(A_i-\alpha_i)^2=\theta_i(1-\theta_i)\le\frac14.
\]
Every realization of
\[
  \hat{\vec x}=\vec x_0+\sum_{i\ne h}A_i\vec w_i
\]
is integral and remains on the objective level. By independence and the zero
means, the cross terms vanish in expectation, so
\[
\begin{aligned}
  \mathbb E\norm{\hat{\vec x}-\vec x^c}^2
  &=\sum_{i\ne h}\mathbb E(A_i-\alpha_i)^2\norm{\vec w_i}^2\\
  &\le\frac14\sum_{i\ne h}(c_h^2+c_i^2)\\
  &=\frac14\bigl(\norm{\vec c}^2+(n-2)c_h^2\bigr)
   =D_h^2.
\end{aligned}
\]
Therefore at least one realization has distance at most $D_h$. Finally,
$\norm{\vec c}^2+(n-2)c_h^2\le 2(n-1)\cmax^2$, which gives
$D_h\le\cmax\sqrt{(n-1)/2}$.

The random rounding is only an existence argument. The dynamic program below
remains deterministic because it searches the entire resulting box; it does not
sample these rounding choices.
\end{proof}

\paragraph{A bounded dynamic program on one objective level.}
Fix a candidate value $V$, let
$s_i=\lfloor Vc_i/\norm{\vec c}^2\rceil$, write
$\vec s=(s_1,\ldots,s_n)$, and set $W=\lceil D_h\rceil+1$. We search only points in the box
\[
  s_i-W\le x_i\le s_i+W
  \qquad (i=1,\dots,n).
\]
At the optimal level this box contains a feasible optimal point. Indeed, let
$\vec x^\circ$ be a closest integer point to $\xc$ on
$\vec c^\top\vec x=\OPT$. By Lemma~\ref{lem:close},
$\norm{\vec x^\circ-\xc}\le D_h$. Moreover, on a fixed objective level,
$\vec x-\xc$ is orthogonal to $\xc$, so minimizing
$\norm{\vec x-\xc}$ is equivalent to minimizing $\norm{\vec x}$. Since some
feasible integer point attains $\OPT$, we have
$\norm{\vec x^\circ}^2\le r^2$. Finally,
$\norm{\vec s-\xc}_\infty\le\tfrac12$, hence
$\norm{\vec x^\circ-\vec s}_\infty\le D_h+\tfrac12\le W$.

Write $x_i=s_i+y_i$. Then $-W\le y_i\le W$ and
$\sum_i c_i y_i=\rho$, where
$\rho=V-\sum_i c_i s_i$. Since
$|s_i-x_i^c|\le\tfrac12$, we have
$|\rho|\le\tfrac12\norm{\vec c}_1$.

For $i=0,\dots,n$ and integer $p$, let $f_i(p)$ be the minimum of
$\sum_{j=1}^i(s_j+y_j)^2$ over choices satisfying
$\sum_{j=1}^i c_jy_j=p$ and $-W\le y_j\le W$, with value $+\infty$ if no such
choice exists. Then
\begin{equation}\label{eq:boxrec}
  f_i(p)
  =
  \min_{-W\le y_i\le W}
  \left[
    f_{i-1}(p-c_i y_i)+(s_i+y_i)^2
  \right],
  \qquad
  f_0(0)=0,
\end{equation}
and $f_0(p)=+\infty$ for $p\ne0$. Every reachable state satisfies
$|p|\le W\norm{\vec c}_1$.

We scan the integers from Lemma~\ref{lem:window} from largest to smallest, run
this dynamic program for each $V$, and stop at the first level for which
$f_n(\rho)\le r^2$.

\begin{theorem}\label{thm:window}
Let $L=\langle\vec c,r^2\rangle$. \BallIP{} can be solved exactly in
\[
  O\!\left(nW^2\norm{\vec c}_1^2\right)\operatorname{poly}(L)
  =
  O\!\left(n^4\cmax^4\right)\operatorname{poly}(L)
\]
bit time. More explicitly, the state and transition count is bounded, up to an
absolute constant, by
\[
  O\!\left(
    n\norm{\vec c}_1^2
    \bigl(1+\norm{\vec c}^2+(n-2)c_h^2\bigr)
  \right).
\]
In particular, the dependence on the radius is polynomial in the binary
encoding length of $r^2$.
\end{theorem}

\begin{proof}
If a scanned value $V$ is larger than $\OPT$, then the dynamic program cannot
return a point of squared norm at most $r^2$, because such a point would be a
feasible integer point with objective value $V>\OPT$.

At $V=\OPT$, the argument above shows that the search box contains an integer
point of squared norm at most $r^2$. Hence $f_n(\rho)\le r^2$. Since
Lemma~\ref{lem:window} places $\OPT$ in the scanned interval, the first accepted
level is exactly $\OPT$.

For a fixed level, the table has
$O(nW\norm{\vec c}_1)$ states and each state considers $O(W)$ transitions.
There are at most $\norm{\vec c}_1+1$ candidate levels, so the total number of
state transitions is $O(nW^2\norm{\vec c}_1^2)$. Since
$W^2=O(1+\norm{\vec c}^2+(n-2)c_h^2)=O(n\cmax^2)$ and
$\norm{\vec c}_1\le n\cmax$, this is $O(n^4\cmax^4)$ arithmetic transitions.
The integers stored in the table have polynomial bit length, so accounting for
integer additions, multiplications, and comparisons gives the stated
$\operatorname{poly}(L)$ factor.

\end{proof}

The two exact algorithms are complementary. The Knapsack DP is
pseudo-polynomial in the numerical value of $r^2$ and polynomial in the binary
encoding length of the objective coefficients, while the objective-level
algorithm is pseudo-polynomial in the objective coefficients and polynomial in
the binary encoding length of $r^2$. Thus \BallIP{} is polynomial-time solvable
whenever either the squared-radius budget or the objective coefficients are
polynomially bounded in the input length. These two bounds leave open the case in which both quantities can be numerically large; the covering-radius algorithm
in the next section addresses part of that case.

\begin{remark}[Approximation via nonlinear knapsack]
Hochbaum~\cite[Section~3]{hochbaum1995} gives an FPTAS for integer nonlinear
knapsack with separable nondecreasing concave profits and a separable
nondecreasing convex packing constraint. By definition, for every
$\varepsilon>0$, an FPTAS computes a feasible solution of value at least
$(1-\varepsilon)\OPT$ in time polynomial in the input encoding length and
$1/\varepsilon$. More generally, for every fixed finite integer $p\ge2$,
Hochbaum's scheme applies directly to $\BallIPp{p}^{+}$ with
$f_i(m)=c_i m$ and $g_i(m)=m^p$, and therefore also to the unrestricted
$\BallIPp{p}$ problem by the sign normalization in
Section~\ref{sec:notation}. Thus the FPTAS is not specific to the Euclidean
norm. For $p=2$, a related greedy bound, which is the additive bound of the
FPTAS of Hochbaum, is also useful for intuition. View the increment
$m-1\to m$ in coordinate $i$ as an item of profit $c_i$ and weight $2m-1$,
and process these increments in nonincreasing order of $c_i/(2m-1)$. If
$\hat{\vec x}$ is the feasible point obtained by stopping when the next
increment would exceed the budget, then
\[
  \OPT^+(\vec c,r^2)-\vec c^\top\hat{\vec x}\le \cmax.
\]
Like the standard rounded fractional knapsack bound, at most one increment is fractional, and the profit of that increment is bounded above by $\cmax$.
\end{remark}



\section{Objective levels and closest-vector geometry}\label{sec:cvp-geometry}

We now organize the feasible set by values of the linear objective.  On one
level $\vec c^\top\vec x=m$, the integer points form a translate of the kernel
lattice of $\vec c^\top$, and the ball cuts that hyperplane in a lower-dimensional
Euclidean ball.  This gives both an algorithmic use of CVP and, in the reverse
direction, a reduction from lattice rank $k$ CVP in dimension $d$ to \BallIP{} in dimension $d+k$.

\subsection{Objective levels and the kernel lattice}\label{sec:cvpalg}

The dynamic program of Subsection~\ref{sec:window} considers the possible objective values one at a time. Lemma~\ref{lem:window} leaves at most
$\norm{\vec c}_1+1\le n\norm{\vec c}_\infty+1$ candidate values, and each
candidate $m$ asks whether its objective level contains an integer point inside
the ball.

In this section, we show that a level is automatically feasible whenever the
radius of the corresponding ball section is at least the covering radius of
the kernel lattice. This bounds the number of levels that a top-down scan
must test before reaching a feasible one.

We continue with the unrestricted formulation used in
Subsection~\ref{sec:window}: according to the convention in
Section~\ref{sec:notation}, throughout this subsection
$\vec x\in\Z^n$, and the cost vector is positive and primitive after
preprocessing. Set
\[
H_0=\{\vec x\in\R^n:\vec c^\top\vec x=0\}, \qquad
  \Lambda_0=\{\vec x\in\Z^n:\vec c^\top\vec x=0\},
  \qquad
  \Lambda_m=\{\vec x\in\Z^n:\vec c^\top\vec x=m\}.
\]
Since $\vec c$ is primitive, a unimodular matrix $U\in\Z^{n\times n}$ (an
integer matrix with determinant $\pm1$) with
$\vec c^\top U=(1,0,\ldots,0)$ can be computed in polynomial time by integer
normal-form algorithms~\cite{kannanbachem1979}. Its first column is a
B\'ezout vector $\vec w_0$ satisfying $\vec c^\top\vec w_0=1$, and its remaining
columns form a basis of $\Lambda_0$. Define
\begin{equation}\label{eq:level-data}
\begin{aligned}
  \boldsymbol{\delta}&=\frac{\vec c}{\norm{\vec c}^2}-\vec w_0,
  &\rho_m^2&=r^2-\frac{m^2}{\norm{\vec c}^2},\\
  d_m&=\operatorname{dist}(m\boldsymbol{\delta},\Lambda_0)
  :=\min_{\vec v\in\Lambda_0}\norm{m\boldsymbol{\delta}-\vec v}.
\end{aligned}
\end{equation}

Then $\Lambda_m=m\vec w_0+\Lambda_0$. When $\rho_m^2\ge0$, the intersection
$\{\vec c^\top\vec x=m\}\cap r\mathbb B^n$ is a possibly degenerate ball
in that hyperplane, with center $(m/\norm{\vec c}^2)\vec c$ and radius $\rho_m$.

The integer points with objective value $m$ are exactly the points of
$\Lambda_m$. Therefore, level $m$ is feasible when some point of
$\Lambda_m$ lies within distance $\rho_m$ of the center of the slice. After
translating by $m\vec w_0$, this is equivalent to asking whether the point
$m\boldsymbol{\delta}$ is within distance $\rho_m$ of $\Lambda_0$. Hence, whenever
$\rho_m^2\ge0$,
\begin{equation}\label{eq:level-equivalence}
  \text{level }m\text{ is feasible}
  \quad\Longleftrightarrow\quad
  d_m\le\rho_m.
\end{equation}
Thus testing whether level $m$ is feasible is a Euclidean closest-vector
query in the fixed $(n-1)$-dimensional kernel lattice $\Lambda_0$, with target
$m\boldsymbol{\delta}$. Only the target changes with $m$.

The covering radius is taken in $H_0=\operatorname{span}(\Lambda_0)$:
\[
  \mu:=\sup_{\vec t\in H_0}\operatorname{dist}(\vec t,\Lambda_0).
\]
It is only an analysis parameter, not an input to the algorithm. Since
$m\boldsymbol{\delta}\in H_0$, every level with $\rho_m\ge\mu$ is feasible.
Thus only levels with $\rho_m<\mu$ can be infeasible; the scan may also test
the first guaranteed feasible level before stopping.

\begin{proposition}[Covering-radius level bound]\label{prop:coverlevels}
Let $\vec c\in\Z^n$ be nonzero and primitive, with $n\ge2$, and let $\mu$
be the covering radius of $\Lambda_0=\Z^n\cap\vec c^\perp$ in
$\vec c^\perp$. If $r\ge\mu$, every integer level
\[
  0\le m\le \norm{\vec c}\sqrt{r^2-\mu^2}
\]
is feasible. Consequently,
\[
  \OPT\ge
  \left\lfloor\norm{\vec c}\sqrt{r^2-\mu^2}\right\rfloor,
\]
and a top-down scan beginning at $\lfloor r\norm{\vec c}\rfloor$ needs to test
at most
\[
  1+\left\lceil
    \frac{\norm{\vec c}\mu^2}
         {r+\sqrt{r^2-\mu^2}}
  \right\rceil
  \le
  1+\left\lceil\frac{\norm{\vec c}\mu^2}{r}\right\rceil
\]
levels before reaching a feasible one.
\end{proposition}

\begin{proof}
If $m\le\norm{\vec c}\sqrt{r^2-\mu^2}$, then
\[
  \rho_m^2
  =r^2-\frac{m^2}{\norm{\vec c}^2}
  \ge\mu^2.
\]
The point $m\boldsymbol{\delta}$ lies in $H_0$, so the definition of covering radius gives
$d_m\le\mu\le\rho_m$. Equation~\eqref{eq:level-equivalence} therefore implies that
level $m$ is feasible. This proves the lower bound on $\OPT$.

Hence a scan from $\lfloor r\norm{\vec c}\rfloor$ downward must stop no later
than $\lfloor\norm{\vec c}\sqrt{r^2-\mu^2}\rfloor$. The number of levels it
can inspect is at most
\[
  1+\left\lceil
    \norm{\vec c}\left(r-\sqrt{r^2-\mu^2}\right)
  \right\rceil
  =
  1+\left\lceil
    \frac{\norm{\vec c}\mu^2}
         {r+\sqrt{r^2-\mu^2}}
  \right\rceil,
\]
and the final inequality follows because the denominator is at least $r$.
\end{proof}

\begin{corollary}[Explicit kernel covering-radius bound]\label{cor:kernel-covering}
Let $h$ be the index used in Lemma~\ref{lem:close}, and recall
\[
  D_h=\frac12\sqrt{\norm{\vec c}^2+(n-2)c_h^2}.
\]
Then
\[
  \mu\le D_h\le \cmax\sqrt{\frac{n-1}{2}}.
\]
Consequently, if $r\ge D_h$, a top-down scan from
$\lfloor r\norm{\vec c}\rfloor$ needs to test at most
\[
  \min\!\left\{
    \norm{\vec c}_1+1,
    1+\left\lceil
      \frac{\norm{\vec c}D_h^2}
           {r+\sqrt{r^2-D_h^2}}
    \right\rceil
  \right\}
  \le
  1+\left\lceil\frac{\norm{\vec c}D_h^2}{r}\right\rceil
\]
levels before reaching a feasible one. In particular,
$r\ge\max\{D_h,\norm{\vec c}D_h^2\}$ is sufficient for this coarse bound to
require at most two tested levels.
\end{corollary}

\begin{proof}
Apply Lemma~\ref{lem:close} with $V=0$ to an arbitrary point of $H_0$.
It gives a point of $\Lambda_0$ within distance $D_h$, so the definition of
covering radius gives $\mu\le D_h$. The displayed uniform estimate is already
part of Lemma~\ref{lem:close}. If $r\ge D_h$, Proposition~\ref{prop:coverlevels}
and the monotonicity of $r-\sqrt{r^2-a^2}$ for $0\le a\le r$ give the second
term in the minimum, while Lemma~\ref{lem:window} gives the first. The final
inequality follows because $r+\sqrt{r^2-D_h^2}\ge r$.
\end{proof}

Proposition~\ref{prop:coverlevels} and Corollary~\ref{cor:kernel-covering} are independent of how the remaining levels
are tested. In particular, one may reuse the dynamic program for a fixed objective level
from Subsection~\ref{sec:window}, replacing its $O(\norm{\vec c}_1)$-level scan by
the potentially shorter bound above. To keep the dependence on the cost
magnitudes polynomial in their binary encoding length at each level, we
instead use an exact-CVP subroutine, with a separate exponential factor in the dimension. The
Voronoi cell consists of the points of $H_0$ at least as close to the origin
as to any other point of $\Lambda_0$. It can be computed once, after which
the deterministic Micciancio--Voulgaris algorithm answers each query in
$2^{O(n)}\operatorname{poly}(L)$ time~\cite{miccianciovoulgaris}.

\paragraph{Top-down scan.}
\begin{enumerate}[nosep,label=\arabic*.]
\item Compute $\vec w_0$, a basis of $\Lambda_0$, and its Voronoi cell.
\item For $m=\lfloor r\norm{\vec c}\rfloor,\lfloor r\norm{\vec c}\rfloor-1,\ldots,0$,
find a closest vector $\vec v\in\Lambda_0$ to $m\boldsymbol{\delta}$.
\item At the first level with
$\norm{\vec v-m\boldsymbol{\delta}}^2\le r^2-m^2/\norm{\vec c}^2$,
return $m\vec w_0+\vec v$ and stop.
\end{enumerate}

\begin{figure}[t]
\centering
\begin{tikzpicture}[>=stealth, node distance=0.55cm and 0.9cm,
  box/.style={draw, rounded corners=2pt, align=center, font=\small, inner sep=5pt, minimum height=0.9cm},
  dec/.style={draw, diamond, aspect=2.6, align=center, font=\small, inner sep=1pt}]
  \node[box] (pre) {preprocessing (once):\\ $\vec w_0$, basis of $\Lambda_0$, Voronoi cell of $\Lambda_0$};
  \node[box, below=of pre] (init) {$m\leftarrow\lfloor r\norm{\vec c}\rfloor$};
  \node[box, below=of init] (cvp) {CVP query in $\Lambda_0$ with target $m\boldsymbol\delta$:\\ closest $\vec v\in\Lambda_0$, $d_m^2=\norm{\vec v-m\boldsymbol\delta}^2$};
  \node[dec, below=of cvp] (test) {$d_m^2\le r^2-\dfrac{m^2}{\norm{\vec c}^2}$?};
  \node[box, right=1.6cm of test] (dec) {$m\leftarrow m-1$};
  \node[box, below=of test] (ret) {return $\vec x=m\vec w_0+\vec v$\\ ($\vec c^\top\vec x=m=\OPT$, $\norm{\vec x}\le r$)};
  \draw[->] (pre) -- (init);
  \draw[->] (init) -- (cvp);
  \draw[->] (cvp) -- (test);
  \draw[->] (test) -- node[above, font=\small] {no} (dec);
  \draw[->] (dec) |- (cvp);
  \draw[->] (test) -- node[right, font=\small] {yes} (ret);
  \node[font=\small, align=left, anchor=west] at (5.3,-3.9) {number of loop passes:\\ $\le\norm{\vec c}_1+1$ (Lemma~\ref{lem:window})\\ $\le1+\lceil\norm{\vec c}\mu^2/r\rceil$ if $r\ge\mu$\\ (Proposition~\ref{prop:coverlevels})};
\end{tikzpicture}
\caption{The top-down scan of Theorem~\ref{thm:covscan}. Each pass is one
exact CVP query in the fixed lattice $\Lambda_0$; by
\eqref{eq:level-equivalence} the test succeeds exactly at feasible levels, so
the first success is the optimal value. The covering radius enters only in the
bound on the number of passes.}
\label{fig:scan-flow}
\end{figure}
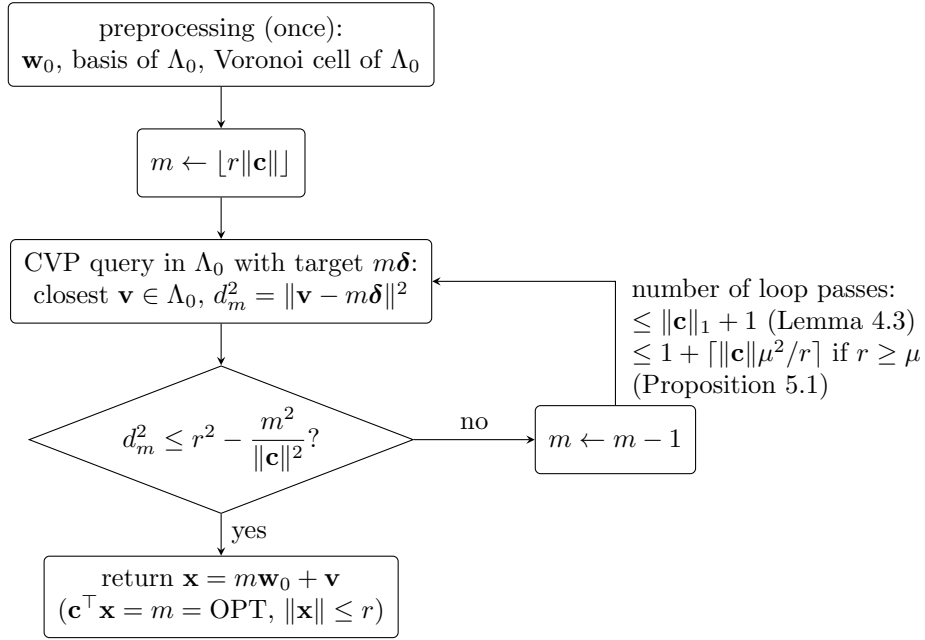

\begin{theorem}[Exact algorithm from the covering-radius bound]\label{thm:covscan}
After the preprocessing of Section~\ref{sec:notation}, let
$\vec c\in\Z_{>0}^n$ be primitive, $n\ge2$, and $r>0$. Let $\mu$ be the
covering radius of $\Lambda_0=\Z^n\cap\vec c^\perp$ in $\vec c^\perp$.
The top-down scan solves \BallIP{} in
\[
  2^{O(n)}K\operatorname{poly}(L),\qquad
  K:=\begin{cases}
    \norm{\vec c}_1+1, & r<\mu,\\[2pt]
    \min\!\left\{\norm{\vec c}_1+1,
      1+\left\lceil\norm{\vec c}\mu^2/r\right\rceil\right\}, & r\ge\mu,
  \end{cases}
\]
where $L=\langle\vec c,r^2\rangle$. In particular, if $r\ge\mu$ and
$\norm{\vec c}\mu^2/r\le2^{O(n)}$ with a uniform constant in the exponent,
the running time is $2^{O(n)}\operatorname{poly}(L)$.
\end{theorem}

\begin{proof}
Lemma~\ref{lem:window} guarantees that the scan reaches a feasible level
within $\norm{\vec c}_1+1$ tests. When $r\ge\mu$,
Proposition~\ref{prop:coverlevels} gives the second bound on their number.
Each test is an exact Euclidean CVP query in the same lattice $\Lambda_0$.
Computing its Voronoi cell once and answering the queries with the
Micciancio--Voulgaris algorithm gives the stated running time. All other
arithmetic has polynomial bit complexity. By
\eqref{eq:level-equivalence}, the returned vector is feasible at the highest
feasible level and hence optimal. The scan does not require $\mu$.
\end{proof}

This complements the two dynamic programs with a different sufficient
condition for efficient exact solution. The Knapsack DP of
Subsection~\ref{sec:dp} is effective when the radius is numerically small, and
the dynamic program of Subsection~\ref{sec:window} when the costs are small.
For fixed $\vec c$ and $r\ge\mu$, the covering-radius bound on the number of
levels instead decreases as $r$ grows.


\subsection{A discriminant obstruction for objective-level lattices}\label{sec:slice-discriminant}

The covering-radius argument above uses the lattice carried by a fixed
objective level. Before turning to the reverse reduction from CVP to \BallIP,
it is natural to ask how general such a lattice can be. More specifically,
can the entire lattice underlying an arbitrary CVP instance be identified
isometrically with the difference lattice of an objective level of
$\mathbb Z^m$ with primitive integer normal vector? In general, the answer is
no: such codimension-one slices (of rank $m-1$) have cyclic discriminant group,
while an integral CVP lattice need not.

\paragraph{Lattice-theoretic definitions.}
For nonzero $\vec u\in\mathbb R^m$, its orthogonal complement is
$\vec u^\perp=\{\vec x\in\mathbb R^m:\vec u^\top\vec x=0\}$. A lattice
$\Lambda\subset\vec u^\perp$ is full-rank in this space if its rank is
$m-1$. For a basis $\vec b_1,\ldots,\vec b_{m-1}$, the covolume
$\operatorname{covol}(\Lambda)$ is the $(m-1)$-dimensional volume of the
fundamental parallelepiped
$\{\sum_{i=1}^{m-1}t_i\vec b_i:0\le t_i<1\}$. The dual lattice is
$\Lambda^*=\{\vec x\in\vec u^\perp:\langle\vec x,\vec y\rangle\in\mathbb Z
\text{ for every }\vec y\in\Lambda\}$. A lattice is integral if the inner
product of any two lattice vectors is an integer. For an integral lattice,
$\Lambda\subseteq\Lambda^*$, and its discriminant group is the finite quotient
$\Lambda^*/\Lambda$. Two lattices $\Lambda_1$ and $\Lambda_2$ are isometric if
there exists a bijection $\varphi:\Lambda_1\to\Lambda_2$ satisfying
$\langle\varphi(\vec x),\varphi(\vec y)\rangle=\langle\vec x,\vec y\rangle$
for all $\vec x,\vec y\in\Lambda_1$. An objective level
$\{\vec x\in\mathbb Z^m:\vec u^\top\vec x=j\}$ is generally a coset; its
difference lattice is $\mathbb Z^m\cap\vec u^\perp$.

Let $\vec u\in\mathbb Z^m$ be primitive and set
\[
    \Lambda_{\vec u}:=\mathbb Z^m\cap \vec u^\perp,
    \qquad
    \Lambda_{\vec u}^*:=\{\vec y\in \vec u^\perp:\vec y^\top\vec x\in\mathbb Z\text{ for every }\vec x\in\Lambda_{\vec u}\}.
\]

\begin{proposition}[Discriminant group of a codimension-one slice]
\label{prop:cyclic-discriminant-short}
For primitive $\vec u\in\mathbb Z^m$,
\[
    \Lambda_{\vec u}^*/\Lambda_{\vec u}\cong
    \mathbb Z/\|\vec u\|_2^2\mathbb Z.
\]
\end{proposition}

\begin{proof}
Since $\vec u$ is primitive, choose $\vec w_0\in\mathbb Z^m$ with
$\vec u^\top\vec w_0=1$ as in Section~\ref{sec:cvpalg}, and put
\[
    \boldsymbol{\delta}:=\frac{\vec u}{\|\vec u\|_2^2}-\vec w_0.
\]
Then $\boldsymbol{\delta}\in \vec u^\perp$, and for every $\vec x\in\Lambda_{\vec u}$ we have
$\boldsymbol{\delta}^\top\vec x=-\vec w_0^\top\vec x\in\mathbb Z$.  Thus
$\boldsymbol{\delta}\in\Lambda_{\vec u}^*$.  Moreover,
$\|\vec u\|_2^2\boldsymbol{\delta}=\vec u-\|\vec u\|_2^2\vec w_0\in\Lambda_{\vec u}$.

If $s\boldsymbol{\delta}\in\Lambda_{\vec u}$, then $s\boldsymbol{\delta}$ is integral, so
$s\vec u/\|\vec u\|_2^2$ is integral.  Hence $\|\vec u\|_2^2$ divides $s u_i$ for every
coordinate $i$. Since $\vec u$ is primitive, there are integers $\lambda_i$ with
$\sum_i\lambda_i u_i=1$ (B\'ezout's identity). Taking the same integer
linear combination of these divisibilities gives $\|\vec u\|_2^2\mid s$.
Therefore the class of $\boldsymbol{\delta}$ has order exactly
$\|\vec u\|_2^2$.

Finally, $\operatorname{covol}(\Lambda_{\vec u})=\|\vec u\|_2$.  Indeed, if
$\vec v_1,\ldots,\vec v_{m-1}$ is a basis of $\Lambda_{\vec u}$, then
$\vec v_1,\ldots,\vec v_{m-1},\vec w_0$ is a basis of $\mathbb Z^m$; its volume is $1$, while
the perpendicular height of $\vec w_0$ above $\vec u^\perp$ is
$|\vec u^\top\vec w_0|/\|\vec u\|_2=1/\|\vec u\|_2$.  Thus the covolume is $\|\vec u\|_2$.
Consequently
$|\Lambda_{\vec u}^*/\Lambda_{\vec u}|=\operatorname{covol}(\Lambda_{\vec u})^2=\|\vec u\|_2^2$.
Since the group already contains an element of this order, it is cyclic.
\end{proof}

For example, a lattice with Gram matrix $2I_k$ (the matrix of pairwise inner
products of a basis) has discriminant group $(\mathbb Z/2\mathbb Z)^k$,
which is noncyclic for $k\ge2$. It therefore
cannot be isometric to $\mathbb Z^m\cap \vec u^\perp$ for any primitive
$\vec u\in\mathbb Z^m$.

This obstruction concerns a global isometric identification of the entire
lattice on an objective level. It therefore does not prevent a bounded
encoding of a CVP instance inside one Euclidean ball. The reduction in the
next subsection uses exactly this distinction: within a bounded region, a
single hyperplane equation cuts out the lifted CVP witnesses, without
requiring the whole hyperplane lattice to be isometric to the source lattice.


\subsection{A reduction from Euclidean CVP to \BallIP}\label{sec:cvp-to-ballip}

We give a reduction from Euclidean CVP to the decision version
of \BallIP. As specified in Section~\ref{sec:notation},
\BallIP{} here is the unrestricted formulation
$\vec x\in\Z^n$ so we do not assume nonnegativity of the variables.

We use the standard integral formulation of Euclidean CVP and assume without
loss of generality that
\[
    B\in\mathbb Z^{d\times k},\qquad \vec t\in\mathbb Z^d,
    \qquad \beta\in\mathbb Z_{\ge 0},
\]
where $B$ has column rank $k$. If $k=0$, the source instance is decided
directly by testing $\|\vec t\|_2^2\le\beta$, so assume $1\le k\le d$. The source
instance asks whether there is $\vec z\in\mathbb Z^k$ such that
$\|B\vec z-\vec t\|_2^2\le\beta$. In particular, $\|B\vec z-\vec t\|_2^2$ is an integer for every
$\vec z\in\mathbb Z^k$. Our construction produces one \BallIP{} instance in dimension
$d+k$.

The construction has four stages. First, we bound the size of any CVP witness
that satisfies the radius constraint. Second, we lift both the error vector
$B\vec z-\vec t$ and the coefficient vector $\vec z$ into a Euclidean ball so that the
integrality of $\norm{B\vec z-\vec t}^2$ creates a one-unit gap. Third, inside this
bounded ball, we combine the lifted graph equations into one integer equation
using a positional encoding. Finally, we translate the resulting hyperplane
section and enlarge the radius so that a one-sided \BallIP{} objective threshold
selects exactly that integer objective level.

\subsubsection{A bounded lift}

We first bound the coefficient vector $\vec z$ of any CVP witness satisfying
the radius bound. Set
\begin{equation}\label{eq:cvp-witness-data}
    M:=\max\Bigl\{1,\max_{i,j}|B_{ij}|\Bigr\},\qquad
    Z:=k\Bigl(k!M^{k-1}(\|\vec t\|_\infty+\beta+1)\Bigr)^2.
\end{equation}

\begin{lemma}\label{lem:cvp-witness-bound}
For an integral CVP instance with $B$ of column rank $k\ge1$, let $Z$ be
as in \eqref{eq:cvp-witness-data}. If $\vec z\in\mathbb Z^k$ satisfies
$\|B\vec z-\vec t\|_2^2\le\beta$, then $\|\vec z\|_2^2\le Z$.
\end{lemma}

\begin{proof}
Choose $k$ linearly independent rows of $B$, and let
$A\in\mathbb Z^{k\times k}$ be the resulting nonsingular submatrix.  If $I$
denotes the selected rows, then
$A\vec z=\vec t_I+(B\vec z-\vec t)_I$.  Since $\|B\vec z-\vec t\|_2\le\sqrt\beta$, every coordinate of the
right-hand side has magnitude at most $\|\vec t\|_\infty+\sqrt\beta$.

Because $A$ is integral and nonsingular, $|\det A|\ge 1$.  Moreover, every
entry of $\operatorname{adj}(A)$ is a cofactor of $A$ and hence has magnitude
at most $(k-1)!M^{k-1}$.  Thus, from
$\vec z=A^{-1}(\vec t_I+(B\vec z-\vec t)_I)$, each coordinate satisfies
\[
\begin{aligned}
    |z_i|
    &\le \frac{1}{|\det A|}
       \sum_{j=1}^k |\operatorname{adj}(A)_{ij}|
       |(\vec t_I+(B\vec z-\vec t)_I)_j| \\
    &\le k!M^{k-1}(\|\vec t\|_\infty+\sqrt\beta)
     \le k!M^{k-1}(\|\vec t\|_\infty+\beta+1).
\end{aligned}
\]
Squaring and summing over $i$ gives $\|\vec z\|_2^2\le Z$.
\end{proof}

For $\vec z\in\mathbb Z^k$, define
\begin{equation}\label{eq:cvp-lift-data}
    S:=Z+1,\qquad R_0^2:=S^2\beta+Z,\qquad
    \vec X_{\vec z}:=\binom{S(B\vec z-\vec t)}{\vec z}\in\mathbb Z^{d+k}.
\end{equation}
We choose $S=Z+1$ so that the entire $\vec z$-block has squared norm at most
$Z$, while increasing the integral quantity $\|B\vec z-\vec t\|_2^2$ by one increases
the squared norm of the first block by $S^2>Z$.  Consequently, every CVP witness satisfies
$\|\vec X_{\vec z}\|_2^2\le R_0^2$.  Conversely, if $\|\vec X_{\vec z}\|_2^2\le R_0^2$, then
\[
    \|B\vec z-\vec t\|_2^2\le \beta+\frac{Z}{S^2}<\beta+1,
\]
and integrality forces $\|B\vec z-\vec t\|_2^2\le\beta$.

It remains to recognize the points of the form $\vec X_{\vec z}$ by a single linear
equation.
Write a general $\vec X\in\mathbb Z^{d+k}$ as $\vec X=(\vec Y,\vec z)$ with
$\vec Y\in\mathbb Z^d$.  The desired points are exactly those satisfying
$\vec Y=S(B\vec z-\vec t)$, that is, $d$ integer equations.  We combine these equations using
distinct powers of a sufficiently large base $Q$.

Choose
\begin{equation}\label{eq:cvp-encoding-data}
\begin{aligned}
    Q&:=(R_0^2+1)(1+SkM)+S\|\vec t\|_\infty+2,
    &\vec a&:=(1,Q,\ldots,Q^{d-1})^\top,\\
    \vec c&:=\binom{\vec a}{-SB^\top\vec a}\in\mathbb Z^{d+k},
    &g&:=-S\vec a^\top\vec t.
\end{aligned}
\end{equation}
Since the first coordinate of $\vec a$, and hence of $\vec c$, is $1$, the vector $\vec c$
is primitive.

\begin{proposition}[CVP as a slice of a ball]\label{prop:bounded-hyperplane-short}
With the parameters in \eqref{eq:cvp-witness-data}--\eqref{eq:cvp-encoding-data},
for every $\vec X=(\vec Y,\vec z)\in\mathbb Z^{d+k}$ with
$\|\vec X\|_2^2\le R_0^2$,
\begin{equation}
    \vec c^\top\vec X=g
    \quad\Longleftrightarrow\quad
    \vec Y=S(B\vec z-\vec t).
    \label{eq:bounded-hyperplane-equiv}
\end{equation}
Hence
\begin{equation}\label{eq:cvp-hyperplane}
\exists \vec z\in\mathbb Z^k:\ \|B\vec z-\vec t\|_2^2\le\beta
\quad\Longleftrightarrow\quad
\exists \vec X\in\mathbb Z^{d+k}:\
\|\vec X\|_2^2\le R_0^2,\ \vec c^\top\vec X=g.
\end{equation}
\end{proposition}

\begin{proof}
Let $\vec e:=\vec Y-SB\vec z+S\vec t\in\mathbb Z^d$.  By construction,
$\vec c^\top\vec X-g=\vec a^\top\vec e$.  If $\|\vec X\|_2^2\le R_0^2$, then every coordinate of
$\vec Y$ and $\vec z$ has magnitude at most $R_0<R_0^2+1$.  Hence, for each $j\in[d]$,
\[
    |e_j|\le |Y_j|+S\sum_{i=1}^k |B_{ji}|\,|z_i|+S|t_j|
    \le (R_0^2+1)(1+SkM)+S\|\vec t\|_\infty=Q-2.
\]

Suppose $\vec a^\top\vec e=0$ but $\vec e\ne\vec0$, and let $h$ be the largest index with
$e_h\ne 0$.  The term $e_hQ^{h-1}$ has magnitude at least $Q^{h-1}$, whereas
the sum of all lower-order terms has magnitude at most
$(Q-2)(Q^{h-1}-1)/(Q-1)<Q^{h-1}$, a contradiction.  Thus $\vec a^\top\vec e=0$ if and only if
$\vec e=\vec0$, proving \eqref{eq:bounded-hyperplane-equiv}.

For \eqref{eq:cvp-hyperplane}, a CVP witness has $\|\vec z\|_2^2\le Z$ by
Lemma~\ref{lem:cvp-witness-bound}, hence its lift $\vec X_{\vec z}$ lies in the ball and on
the hyperplane.  Conversely, any integer point in the ball on the hyperplane
is $\vec X_{\vec z}$ by \eqref{eq:bounded-hyperplane-equiv}, and the choice $S=Z+1$ then
gives $\|B\vec z-\vec t\|_2^2\le\beta$.
\end{proof}

\subsubsection{From a hyperplane equality to a \BallIP{} threshold}
At this point the CVP instance has already been encoded exactly as the bounded
slice

$$
    H\cap R_0\mathbb B^{d+k},
    \qquad
    H:=\{\vec X:\vec c^\top\vec X=g\}.
$$

The only remaining issue is that \BallIP{} uses a one-sided objective threshold,
not an equality. Set
\begin{equation}
q:=R_0^2+2|g|+1,
\qquad
\Gamma:=g+q|\vec c|_2^2,
\qquad
r^2:=R_0^2+2qg+q^2|\vec c|_2^2.
\label{eq:q-gamma-r}
\end{equation}
We translate this slice in the normal direction of $H$ by making the
affine change of coordinates

$$
    \vec W=\vec X+q\vec c.
$$

In the $\vec W$-coordinates, the hyperplane becomes
$\vec c^\top\vec W=\Gamma$, and the original ball becomes the ball of radius
$R_0$ centered at $q\vec c$. We then replace this shifted ball by the
origin-centered ball $r\mathbb B^{d+k}$. The two balls are different, but they
cut the translated hyperplane in the same slice. Thus they contain exactly
the same points with $\vec c^\top\vec W=\Gamma$. The choice of $q$ ensures
that no point with $\vec c^\top\vec W\ge\Gamma+1$ lies in
$r\mathbb B^{d+k}$.

Figure~\ref{fig:translate} illustrates the construction. Translating by $q\vec c$ preserves the section at $\vec c^\top\vec W=\Gamma$, and the larger origin-centered ball contains has the same intersection with the translated hyperplane while excluding the larger integer objective levels, beginning with $\Gamma+1$. Real values strictly between $\Gamma$ and $\Gamma+1$ are irrelevant because $\vec c^\top\vec W$ is integral for integer $\vec W$.

\begin{figure}[t]
\centering
\begin{tikzpicture}[>=latex, line join=round, line cap=round, scale=1.05]
\pgfmathsetmacro{\RO}{1.0}
\pgfmathsetmacro{\qlen}{2.8}
\pgfmathsetmacro{\angc}{45}
\pgfmathsetmacro{\tgam}{2.2}
\pgfmathsetmacro{\tgamone}{2.45}
\pgfmathsetmacro{\half}{sqrt(\RO*\RO-(\tgam-\qlen)*(\tgam-\qlen))}
\pgfmathsetmacro{\rnew}{sqrt(\tgam*\tgam+\half*\half)}
\coordinate (O) at (0,0);
\coordinate (C) at (\angc:\qlen);
\coordinate (M) at (\angc:\tgam);
\coordinate (Mone) at (\angc:\tgamone);
\coordinate (P) at ($(M)+(\angc+90:\half)$);
\coordinate (Q) at ($(M)+(\angc-90:\half)$);
\fill[gray!15] (C) circle (\RO);
\draw[thin] (C) circle (\RO);
\draw[dashed, thin] (O) circle (\RO);
\draw[semithick] (O) circle (\rnew);
\draw[->, thin] (-\rnew-0.6,0) -- (\qlen+\RO+0.7,0) node[below] {$W_1$};
\draw[->, thin] (0,-\rnew-0.6) -- (0,\qlen+\RO+0.7) node[left] {$W_2$};
\draw[->, thick] (O) -- (C);
\draw[thin] ($(M)+(\angc+90:2.3)$) -- ($(M)+(\angc-90:2.3)$);
\node[font=\small, anchor=south east] at ($(M)+(\angc+90:2.3)$) {$\vec c^\top\vec W=\Gamma$};
\draw[densely dashed, thin] ($(Mone)+(\angc+90:2.3)$) -- ($(Mone)+(\angc-90:2.3)$);
\node[font=\small, anchor=south west] at ($(Mone)+(\angc+90:2.0)$) {$\vec c^\top\vec W=\Gamma+1$};
\draw[line width=2pt] (P) -- (Q);
\fill (P) circle (1.6pt); \fill (Q) circle (1.6pt);
\draw[densely dotted] (O) -- (P) node[pos=0.5, above left=-1pt, font=\small] {$r$};
\draw[densely dotted] (C) -- ($(C)+(\angc+135:\RO)$) node[pos=0.55, above=-1pt, font=\small] {$R_0$};
\draw[densely dotted] (O) -- (\angc+200:\RO) node[pos=0.6, left=-1pt, font=\small] {$R_0$};
\fill (O) circle (1.6pt); \node[below left=-1pt] at (O) {$\vec 0$};
\fill (C) circle (1.6pt); \node[right=2pt, font=\small] at (C) {$q\vec c$};
\node[font=\small, anchor=west] at ($(C)+(0.75,0.95)$) {$B(q\vec c,R_0)$};
\node[font=\small, anchor=north east] at (225:\rnew+0.05) {$B(\vec 0,r)$};
\node[font=\small, anchor=north east] at (-0.2,-\RO+0.05) {$B(\vec 0,R_0)$};
\node[font=\small, anchor=west] at ($(M)+(\angc-90:1.05)+(0.25,-0.35)$) {common slice};
\end{tikzpicture}
\caption{The translation step that replaces the hyperplane equality by a \BallIP{} threshold. The CVP instance is represented by the slice of $B(\vec 0,R_0)$ at $\vec c^\top\vec X=g$. Translating by $q\vec c$ moves this slice to the level $\vec c^\top\vec W=\Gamma$, where the translated ball and the new origin-centered ball $B(\vec 0,r)$ have the same intersection. The new ball is chosen to exclude the next integer level $\Gamma+1$, so within $B(\vec 0,r)$ the threshold $\vec c^\top\vec W\ge\Gamma$ isolates exactly the desired slice. The figure is schematic, with the spacing between $\Gamma$ and $\Gamma+1$ exaggerated.}
\label{fig:translate}
\end{figure}
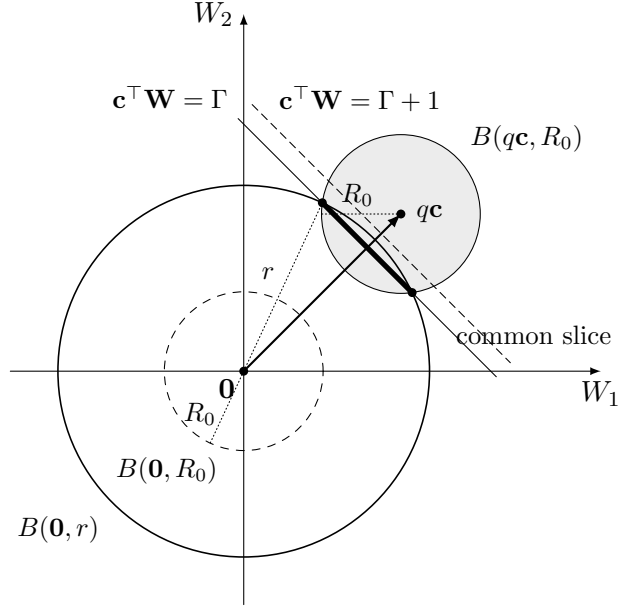

Since the first coordinate of $\vec c$ is $1$, we have $\|\vec c\|_2^2\ge 1$, so
$\Gamma\ge -|g|+q>0$.  Moreover,
\[
    r^2\ge R_0^2-2q|g|+q^2
    =R_0^2+q(q-2|g|)>0,
\]
where the last inequality uses $q=R_0^2+2|g|+1$.

\begin{lemma}[Replacing equality by a threshold]\label{lem:level-isolation-short}
With $q,\Gamma,r^2$ as in \eqref{eq:q-gamma-r}, translation by $q\vec c$
gives a bijection between the integer points on
$\vec c^\top\vec X=g$ and those on $\vec c^\top\vec W=\Gamma$.  Moreover,
\begin{equation}
\{\vec W:\vec c^\top\vec W=\Gamma,\ \|\vec W-q\vec c\|_2^2\le R_0^2\}
=
\{\vec W:\vec c^\top\vec W=\Gamma,\ \|\vec W\|_2^2\le r^2\}.
\label{eq:same-section}
\end{equation}
Finally, no real point with $\vec c^\top\vec W\ge\Gamma+1$ lies in
$B(0,r)$.
\end{lemma}

\begin{proof}
The first statement is immediate from
$\vec c^\top(\vec X+q\vec c)=g+q\|\vec c\|_2^2=\Gamma$, and translation by the integral vector
$q\vec c$ preserves the integer lattice.

For the second statement, restrict to $\vec c^\top\vec W=\Gamma$.  Expanding
$\|\vec W-q\vec c\|_2^2$ shows that
\[
    \|\vec W-q\vec c\|_2^2\le R_0^2
    \quad\Longleftrightarrow\quad
    \|\vec W\|_2^2\le R_0^2+2q\Gamma-q^2\|\vec c\|_2^2=r^2,
\]
which is exactly \eqref{eq:same-section}.  Thus the shifted old ball and the
new origin-centered ball cut out precisely the same slice of the translated
hyperplane.

It remains to exclude larger objective values.  If $\vec c^\top\vec W\ge\Gamma+1$,
Cauchy--Schwarz gives $\|\vec W\|_2^2\ge(\Gamma+1)^2/\|\vec c\|_2^2$.  On the other
hand,
\[
    (\Gamma+1)^2-\|\vec c\|_2^2r^2
    =(g+1)^2+\|\vec c\|_2^2(2q-R_0^2)>0,
\]
so $(\Gamma+1)^2/\|\vec c\|_2^2>r^2$.  Hence every such $\vec W$ lies outside the ball.
\end{proof}

Since $\vec c^\top\vec W$ is integral for $\vec W\in\mathbb Z^{d+k}$,
Lemma~\ref{lem:level-isolation-short} implies
that inside $B(0,r)$ the conditions $\vec c^\top\vec W\ge\Gamma$ and
$\vec c^\top\vec W=\Gamma$ are equivalent: any strictly larger value is at least
$\Gamma+1$ and therefore lies outside the ball.  Combining this with
Proposition~\ref{prop:bounded-hyperplane-short} gives the reduction.

\begin{theorem}\label{thm:main}
Given a Euclidean CVP instance
$B\in\mathbb Z^{d\times k}$, $\vec t\in\mathbb Z^d$, and
$\beta\in\mathbb Z_{\ge 0}$, with $B$ of column rank $k$, one can construct in
deterministic polynomial time an integer vector
$\vec c\in\mathbb Z^{d+k}$ and integers $r^2\in\mathbb Z_{\ge 0}$ and
$\Gamma\in\mathbb Z$ such that
\[
\exists \vec z\in\mathbb Z^k:\ \|B\vec z-\vec t\|_2^2\le\beta
\quad\Longleftrightarrow\quad
\exists \vec W\in\mathbb Z^{d+k}:\
\|\vec W\|_2^2\le r^2,\ \vec c^\top\vec W\ge\Gamma.
\]
All output data have polynomial bit length.  In particular, a full-rank
$n$-dimensional CVP instance maps to a \BallIP{} instance in dimension $2n$.
\end{theorem}

\begin{proof}
Proposition~\ref{prop:bounded-hyperplane-short} and
Lemma~\ref{lem:level-isolation-short} establish the equivalence. It remains to verify that the
reduction runs in polynomial time. Lemma~\ref{lem:cvp-witness-bound} gives $Z$ with polynomial bit
length, hence the same is true of $S$ and $R_0^2$.  $Q$ also has
polynomial bit length.  Although the entries of
$\vec a=(1,Q,\ldots,Q^{d-1})$ may be numerically large, their bit length is at most
$O(d\log Q)$; consequently $\vec c$ and $g$ have polynomial bit length as well.
The same then holds for $q$, $\Gamma$, and $r^2$ by
\eqref{eq:q-gamma-r}.  The construction uses only polynomially many
integer-arithmetic operations, and the dimension is $d+k$ throughout.
\end{proof}

\begin{remark}[Fine-grained consequence]\label{rem:cvp-ballip-hardness}
Aggarwal, Gupta, Morolia, and Zhang~\cite[Theorem~1.1]{AGM25} prove an ETH
lower bound for constant-factor gap-$\mathrm{CVP}_2$: distinguish distance at
most $R$ from distance greater than $\eta R$, for a fixed $\eta>1$.
Their hard lattices in ambient dimension $m$ are full-rank and are given as
the integer combinations of the columns of $[M\;2I_m]$, a generating set
rather than a basis. Integer normal-form algorithms convert it into
a basis in polynomial time~\cite{kannanbachem1979}. Since their squared YES
radius is rational and all squared distances are integers, taking its floor
gives the integral budget required by Theorem~\ref{thm:main}. Exact CVP at
that budget solves the gap instance, and our reduction produces a \BallIP{}
instance in dimension $2m$. Thus exact \BallIP{} in dimension $n$ admits no
$2^{o(n)}\operatorname{poly}(L)$-time algorithm unless ETH fails.

We do not obtain the sharper $2^{(1-\varepsilon)m}$-type lower bounds
associated with the Strong Exponential Time Hypothesis (SETH), which posits
that for every $\varepsilon>0$ there is a clause width $k$ for which $k$-SAT
cannot be solved in $2^{(1-\varepsilon)n}$ time: the quantitative
SETH-hardness result for CVP in~\cite{BGS17} does not cover $p=2$.
\end{remark}


\section{Proximity bounds and the limits of boundary enumeration}\label{sec:geometry}\label{sec:boundary}

The preceding algorithms use two properties of the Euclidean ball:
the constraint $\sum_i x_i^2\le r^2$ is separable across coordinates, and its
intersection with an objective hyperplane is again a Euclidean ball. General
integer programming over convex sets already admits dimension-dependent
algorithms; the randomized bound of Reis and Rothvoss~\cite{reisrothvoss} has
$(\log n)^{O(n)}$ dependence on the dimension, up to polynomial factors. We
therefore ask whether the additional structure of the ball yields a better
exact algorithm for \textsc{BallIP}, with the nonpolynomial factor depending
only on the dimension.

The continuous optimum of \textsc{BallIP} is
$r\vec c/\norm{\vec c}$. A natural idea is to develop an enumeration algorithm
around this point. We bound the coordinates of an optimal integer point
recursively and give an exact search. We then show that explicitly listing
all points of the resulting boundary cap can require radius-dependent work,
already in dimension two; this is not a lower bound for more selective searches.

By the sign-symmetry normalization in Section~\ref{sec:notation}, we may,
without loss of generality throughout this section, take the active costs to
be positive and an optimal solution to satisfy $\vec x\in\Z_{\ge0}^n$.
Accordingly, we work here with \textsc{BallIP}$^+$.
Throughout this section, order the coordinates so that
$c_1\ge\cdots\ge c_n$, and write
$\vec c_{k:n}=(c_k,\ldots,c_n)$ and
$L=\langle\vec c,r^2\rangle$. Rounding the continuous optimum down
coordinatewise gives
\begin{equation}\label{eq:boundary-window}
  r\norm{\vec c}-\norm{\vec c}_1
  \le \OPT^+(\vec c,r^2)
  \le r\norm{\vec c}.
\end{equation}

\subsection{Coordinate bounds}
\label{sec:coordinate-bounds}

Fix $x_1,\ldots,x_{k-1}$ as in an optimal integer solution, and let
$R_k^2=r^2-\sum_{j<k}x_j^2$ be the remaining squared radius. Set
$s_k=\norm{\vec c_{k:n}}$. If $x_k=t\in[0,R_k]$, then the largest possible
contribution of $x_k,\ldots,x_n$ in the continuous relaxation is
$F_k(t)=c_k t+\norm{\vec c_{k+1:n}}\sqrt{R_k^2-t^2}$, where the second term is
zero when $k=n$. This function is maximized at
$t_k^*=c_kR_k/s_k$, with $F_k(t_k^*)=R_ks_k$.

\begin{lemma}[Quadratic loss]\label{lem:profileloss}
If $R_k>0$, then for every $t\in[0,R_k]$,
\[
  F_k(t_k^*)-F_k(t)
  \ge
  \frac{s_k}{2R_k}(t-t_k^*)^2.
\]
When $R_k=0$, the only possible coordinate is $t=t_k^*=0$.
\end{lemma}

\begin{proof}
Assume $R_k>0$. Write $c=c_k$, and let
$a=\norm{\vec c_{k+1:n}}$ when $k<n$ and $a=0$ when $k=n$. Then
$s_k=\sqrt{c^2+a^2}$ and $t_k^*=cR_k/s_k$. Define
\[
  \vec u(t)=\bigl(t,\sqrt{R_k^2-t^2}\bigr),
  \qquad
  \vec u^*=\frac{R_k}{s_k}(c,a).
\]
Both vectors have norm $R_k$, and $\vec u(t_k^*)=\vec u^*$. Moreover,
\[
  F_k(t)=(c,a)^\top\vec u(t)
  =\frac{s_k}{R_k}(\vec u^*)^\top\vec u(t).
\]
Since $F_k(t_k^*)=R_ks_k$, it follows that
\begin{align*}
  F_k(t_k^*)-F_k(t)
  &=\frac{s_k}{R_k}
    \bigl(R_k^2-(\vec u^*)^\top\vec u(t)\bigr)\\
  &=\frac{s_k}{2R_k}\norm{\vec u(t)-\vec u^*}^2\\
  &\ge\frac{s_k}{2R_k}(t-t_k^*)^2.
\end{align*}
The second equality uses
$\norm{\vec u(t)-\vec u^*}^2
=2R_k^2-2(\vec u^*)^\top\vec u(t)$; the final inequality retains only
the first coordinate of the squared norm.
\end{proof}

\begin{theorem}[Coordinate interval bound]\label{thm:coordinterval}
Every optimal solution satisfies
\[
  |x_k-t_k^*|
  \le
  \sqrt{\frac{2R_k\norm{\vec c_{k:n}}_1}{s_k}}
  \le \sqrt2\,\sqrt{R_k}\,(n-k+1)^{1/4}.
\]
\end{theorem}

\begin{proof}
The suffix of an optimal solution is optimal for the residual problem with
costs $\vec c_{k:n}$ and squared radius $R_k^2$: replacing it by a better
suffix would improve the original solution. Rounding the residual continuous
optimizer down coordinatewise therefore gives a suffix value of at least
$R_ks_k-\norm{\vec c_{k:n}}_1$. Since $F_k(x_k)$ is at least this suffix
value, we have
$F_k(t_k^*)-F_k(x_k)\le\norm{\vec c_{k:n}}_1$.
Lemma~\ref{lem:profileloss} gives the first bound when $R_k>0$; the case
$R_k=0$ is immediate. The second follows from
$\norm{\vec c_{k:n}}_1\le\sqrt{n-k+1}\,\norm{\vec c_{k:n}}$.
\end{proof}

\begin{theorem}[Boundary enumeration]\label{thm:enumeration}
For an absolute constant $C$, \textsc{BallIP}$^+$ can be solved exactly in
\[
  \bigl(C(1+\sqrt r\,n^{1/4})\bigr)^{n-1}\operatorname{poly}(L)
\]
time.
\end{theorem}

\begin{proof}
At each feasible prefix $x_1,\ldots,x_{k-1}$, compute $R_k$ and $t_k^*$,
and set $a_k=\lfloor t_k^*\rfloor$ and
$H_k=\lceil3\sqrt{R_k\norm{\vec c_{k:n}}_1/s_k}\rceil$.
For $k<n$, enumerate the integers in
$[a_k-H_k,a_k+H_k+1]\cap[0,\lfloor R_k\rfloor]$.
Since $3>\sqrt2$, Theorem~\ref{thm:coordinterval} guarantees that every
prefix of an optimal solution survives this search; no optimality assumption
is needed for the other prefixes. Once $x_1,\ldots,x_{n-1}$ are fixed,
take $x_n=\lfloor\sqrt{r^2-\sum_{k<n}x_k^2}\rfloor$, which is best because
$c_n>0$. Return the best resulting point.

There are at most $2H_k+2\le6\sqrt r\,n^{1/4}+4$ choices at each of the
first $n-1$ coordinates, so there are at most
$\bigl(C(1+\sqrt r\,n^{1/4})\bigr)^{n-1}$ complete candidates.
All interval endpoints can be computed exactly in polynomial bit time:
$a_k$ is found by comparing integer squares, and $H_k$ is the least
nonnegative integer satisfying
$H_k^4s_k^2\ge81R_k^2\norm{\vec c_{k:n}}_1^2$.
The search depth and all other arithmetic contribute only a polynomial factor.
\end{proof}

\begin{remark}[Comparison with enumerating the integer hull]\label{rem:vertexcount}
Let $P_r=\operatorname{conv}(\Z^n\cap r\mathbb B^n)$. A linear objective
has an optimum at a vertex of $P_r$, so enumerating all its vertices and
comparing their objective values is one exact approach. For fixed $n\ge2$,
B\'ar\'any and Larman~\cite[Theorem~1]{baranylarman} show that $P_r$ has
$\Theta_n(r^{\,n(n-1)/(n+1)})$ vertices. Any method that explicitly visits
every vertex must therefore visit that many points. In contrast, the proof of
Theorem~\ref{thm:enumeration} evaluates at most $O_n(r^{(n-1)/2})$ complete
candidates, and $(n-1)/2<n(n-1)/(n+1)$ for every $n\ge2$. Thus its
candidate count grows more slowly with $r$ than the full vertex count, already
in the plane ($r^{1/2}$ versus $r^{2/3}$). This compares point counts with
full vertex enumeration; it is not a lower bound for methods that select
vertices using the objective or use an implicit representation.
\end{remark}

\subsection{Radius dependence of explicit cap enumeration}\label{sec:cap-size}

Equation~\eqref{eq:boundary-window} bounds the objective loss from the
continuous optimum by $\norm{\vec c}_1$, but this does not bound the Euclidean
distance to the continuous optimum independently of $r$. Rotate the
coordinates so that $\vec c/\norm{\vec c}$ is the last coordinate direction
and write an optimal point as $\vec y=(\bar{\vec y},y_n)$. Set
$\delta=\norm{\vec c}_1/\norm{\vec c}$. Since the objective is nonnegative,
$y_n\ge\max\{0,r-\delta\}$, and hence
\[
  \norm{\bar{\vec y}}^2
  \le r^2-\max\{0,r-\delta\}^2
  \le 2r\delta.
\]
Thus an objective loss of at most $\norm{\vec c}_1$ can still permit an
orthogonal displacement of order $\sqrt r$.

This growth occurs already for $n=2$ and $\vec c=(1,1)$. For integers
$m\ge4$, set $r_m^2=m(m+1)/2$. Since $m<\sqrt2\,r_m<m+1$, every feasible
integer point satisfies $x_1+x_2\le m$. On the level $x_1+x_2=m$, write
$x_1=(m+d)/2$ and $x_2=(m-d)/2$. These coordinates are nonnegative integers
exactly when $d\equiv m\pmod2$ and $|d|\le m$, and the ball constraint is
\[
  x_1^2+x_2^2=\frac{m^2+d^2}{2}\le r_m^2
  \quad\Longleftrightarrow\quad d^2\le m.
\]
There are $\Theta(\sqrt m)=\Theta(\sqrt{r_m})$ such points, all optimal and
in the boundary cap from~\eqref{eq:boundary-window}. Yet the balanced point
$(\lfloor m/2\rfloor,\lceil m/2\rceil)$ is also optimal and lies within
distance $1$ of the continuous optimizer.

All these optimal points lie on the single edge of $P_{r_m}$ supported by
$x_1+x_2=m$; only its two endpoints are vertices. The example therefore
bounds the work needed to list every point of the cap, or every optimal point,
not the work needed to find one optimum. It gives no lower bound for selecting
only vertices, pruning the search, or representing the cap implicitly.

\section{Discussion: Comparison with ellipsoids}\label{sec:ellipsoid-discussion}

Consider
\begin{equation}\label{eq:ellip}
  \textsc{EllipsoidIP}:\qquad
  \max\bigl\{\vec f^\top\vec z:\ \norm{B\vec z-\vec t}^2\le\beta,\ \vec z\in\Z^k\bigr\},
\end{equation}
where $B\in\Z^{d\times k}$ has column rank $k$, $\vec t\in\Z^d$,
$\beta\in\Z_{\ge0}$, and $\vec f\in\Z^k\setminus\{\vec0\}$.
Divide $\vec f$ by the gcd of its entries and rescale objective values, so
that $\vec f$ is primitive. \BallIP{} is the special case $B=I_k$,
$\vec t=\vec0$, and $\beta=r^2$. We discuss which of the results for \BallIP{} hold for ellipsoids.

\paragraph{Fixed objective levels remain closest-vector problems.}
For $k\ge2$, let $K_{\vec f}\in\Z^{k\times(k-1)}$ be a lattice basis of
$\Lambda_{\vec f}=\{\vec y\in\Z^k:\vec f^\top\vec y=0\}$.
Choose $\vec z_j\in\Z^k$ with $\vec f^\top\vec z_j=j$.
Every integer point on level $j$ has the form
$\vec z_j+K_{\vec f}\vec y$, with $\vec y\in\Z^{k-1}$. Thus level $j$ is
feasible exactly when
\[
  \operatorname{dist}\bigl(\vec t-B\vec z_j,
     BK_{\vec f}\Z^{k-1}\bigr)^2\le\beta.
\]
Only the target changes with $j$: the lattice is fixed, and its CVP
preprocessing can be reused. By Section~\ref{sec:cvp-to-ballip}, each query
also reduces to one \BallIP{} decision instance of dimension at most
$d+k-1$. For $k=1$, each level contains a single integer point and is
checked directly.

\paragraph{The dynamic programs require additional structure.}
Both dynamic programs in Section~\ref{sec:exact} use coordinatewise
separability. For general $B$, cross terms in $\norm{B\vec z-\vec t}^2$
prevent the same recurrences. The objective-level algorithm also uses short
kernel vectors; their ellipsoid counterparts
$B(f_h\vec e_i-f_i\vec e_h)$, with $f_h\ne0$, have lengths depending on $B$.
Thus the bounds in Section~\ref{sec:window} cannot simply be obtained by
replacing $\vec c$ with $\vec f$.

\paragraph{Diagonal quadratic forms retain the budget algorithm.}
Suppose $B^\top B=\operatorname{diag}(a_1,\ldots,a_k)$ and set
$b_i=(B^\top\vec t)_i$. Then
\[
  \norm{B\vec z-\vec t}^2
  =\norm{\vec t}^2+\sum_{i=1}^k q_i(z_i),
  \qquad q_i(m)=a_i m^2-2b_i m .
\]
After reflecting coordinates if necessary, assume $f_i\ge0$. For each $i$,
choose an integer minimizer $z_i^0$ of $q_i$, taking the larger minimizer in
case of a tie. Any $z_i<z_i^0$ is then dominated by $z_i^0$, so write
$z_i=z_i^0+x_i$ with $x_i\in\Z_{\ge0}$. If
$D:=\beta-\norm{B\vec z^0-\vec t}^2<0$, the instance is infeasible; otherwise
it becomes, up to the constant $\vec f^\top\vec z^0$,
\[
  \max\Bigl\{\sum_i f_i x_i:\ \sum_i g_i(x_i)\le D,\ x_i\in\Z_{\ge0}\Bigr\},
  \qquad
  g_i(x)=q_i(z_i^0+x)-q_i(z_i^0).
\]
Each $g_i$ is nonnegative, integer-valued, convex, and nondecreasing, and
$g_i(x)\le D$ restricts $x$ to $O(1+\sqrt\beta)$ values. Thus this is exactly
Hochbaum's integer nonlinear-knapsack setting~\cite[Section~1]{hochbaum1995},
so their pseudo-polynomial knapsack algorithm applies directly.

\paragraph{Unimodular ellipsoids inherit the exact ball algorithms.}
Suppose $d=k$ and $B$ is unimodular, meaning $\det B=\pm1$.
Then $B^{-1}$ is integral, so $\vec x=B\vec z-\vec t$ is a bijection of
$\Z^k$, with inverse $\vec z=B^{-1}(\vec x+\vec t)$. In these coordinates,
\[
  \norm{\vec x}^2\le\beta,
  \qquad
  \vec f^\top\vec z=(B^{-\top}\vec f)^\top\vec x
     +\vec f^\top B^{-1}\vec t.
\]
Thus we obtain \BallIP{} with objective $\vec c=B^{-\top}\vec f$ and squared
radius $\beta$. $\vec f^\top B^{-1}\vec t$ is a constant, hence it does not affect the optimal solution.
For unimodular $B$, all exact ball algorithms apply. Their bounds use the transformed coefficients, which may be much
larger than those of $\vec f$.

\paragraph{Ellipsoid feasibility is strongly NP-hard.}
Without the objective, \eqref{eq:ellip} asks whether
$$\operatorname{dist}(\vec t,B\Z^k)^2\le\beta$$ precisely the Euclidean CVP
decision problem. Its NP-hardness already holds for polynomially bounded
integer data~\cite{BGS17}. Therefore, unless $\mathrm{P}=\mathrm{NP}$, the ball's
pseudo-polynomial guarantees cannot extend to arbitrary ellipsoids, even
with polynomially bounded $\beta$ and a unit-coordinate objective.

\section{Conclusion}\label{sec:open}\label{sec:discussion}

We have shown that for every fixed integer $p\ge2$, the decision version of
$\BallIPp{p}$ is NP-complete. For the Euclidean case, the optimization problem
also admits complementary pseudo-polynomial algorithms and a close connection
with Euclidean CVP. We conclude with the main open question.

\textbf{Can the special ball geometry beat general convex integer programming?}
General integer-programming methods already give randomized $(\log n)^{O(n)}$ dependence
on the dimension, up to polynomial factors~\cite{reisrothvoss}; for the
quasiconvex-polynomial formulation, a deterministic
$2^{O(n\log n)}\operatorname{poly}(L)$ bound is also available
\cite{hk2010}.  Section~\ref{sec:boundary} shows that explicitly listing all points of a boundary cap requires radius dependent work.  By contrast, Theorem~\ref{thm:covscan} gives
an exact
\[
  2^{O(n)}\left(1+\frac{\norm{\vec c}\mu^2}{r}\right)
  \operatorname{poly}(L)
\]
algorithm whenever $r\ge\mu$, and therefore a
$2^{O(n)}\operatorname{poly}(L)$ algorithm whenever
$\norm{\vec c}\mu^2/r\le 2^{O(n)}$.  Does \BallIP{} admit an exact
$2^{O(n)}\operatorname{poly}(L)$ algorithm for all instances?  The reduction
of Section~\ref{sec:cvp-to-ballip} rules out
$2^{o(n)}\operatorname{poly}(L)$ time under ETH.  More precisely, it maps a
rank-$k$ Euclidean CVP instance in dimension $d$ to a \BallIP{} instance
of dimension $d+k\le2d$.  Thus any \BallIP{} running-time bound $T(n)$
yields, up to polynomial factors, a Euclidean-CVP bound $T(2d)$; improving the exponential dependence for \BallIP{} would therefore improve the
induced CVP bound. Abboud and Kumar~\cite[Theorem~1.3]{AK25} obtain a sub-$2^n$
algorithm for a restricted binary-coefficient CVP problem; our reduction
concerns unrestricted CVP. For a full-rank input, the reduction maps
dimension $d$ to dimension $2d$, so a \BallIP{} bound $a^n$ induces a CVP bound
$(a^2)^d$. Beating $2^d$ through this reduction would therefore require
$a<\sqrt2$, not merely $a<2$.

\bibliographystyle{plainnat}
\bibliography{references}

\end{document}